\documentclass[english,letterpaper,11pt,reqno]{amsart}
\usepackage[backend=bibtex,style=alphabetic,maxnames=6]{biblatex}
\usepackage{appendix}
\usepackage{amsbsy}
\usepackage{scalerel}
\usepackage{tcolorbox}
\usepackage{amsfonts}
\usepackage{amsmath}
\usepackage{amssymb}
\usepackage{amsthm}
\usepackage{graphicx}
\usepackage{ifthen}
\usepackage{textcomp}
\usepackage{enumitem, color, amssymb}
\usepackage{dsfont}
\usepackage{mathrsfs}
\usepackage{calrsfs}
\usepackage[bookmarksnumbered,colorlinks]{hyperref}
\hypersetup{colorlinks=true, linkcolor=blue, citecolor=blue}
\usepackage{cancel}
\usepackage{setspace}

\newcommand{\lra}{\longrightarrow}
\newcommand{\uu}{\mathscr{U}}
\newcommand\gpw{h_{\scriptscriptstyle \vep}}
\newcommand\hpw{h_{\scalebox{.4}{PW}}}

\newcommand{\RR}{\mathbb{R}}

\newcommand{\vep}{\varepsilon}

\makeatletter
\newcommand*{\defeq}{\mathrel{\rlap{%
                     \raisebox{0.25ex}{$\m@th\cdot$}}%
                     \raisebox{-0.25ex}{$\m@th\cdot$}}%
                     =}
\makeatother

\makeatletter
\newcommand*\owedge{\mathpalette\@owedge\relax}
\newcommand*\@owedge[1]{%
  \mathbin{%
    \ooalign{%
      $#1\m@th\bigcirc$\cr
      \hidewidth$#1\m@th\wedge$\hidewidth\cr
    }%
  }%
}
\makeatother

\newtheorem{thm}{Theorem}
\newtheorem{lemma}{Lemma}
\newtheorem{cor}{Corollary}

\newtheorem{defn}{Definition}
\newtheorem{prop}{Proposition}
\newtheorem*{definition-non}{Definition}
\newtheorem*{theorem-non}{Theorem}
\newtheorem*{proposition-non}{Proposition}
\newtheorem*{lemma-non}{Lemma}
\newtheorem*{corollary-non}{Corollary}

\newcommand{\beqa}{\begin{eqnarray}}
\newcommand{\beq}{\begin{equation}}
\newcommand{\eeqa}{\end{eqnarray}}
\newcommand{\eeq}{\end{equation}}

\newcommand\imp{\hspace{.2in}\Rightarrow\hspace{.2in}}

\newcommand\Aij{A^{\scalebox{0.6}{\emph{$\gamma$}}}_{ij}}

\newcommand\Aii{A^{\scalebox{0.6}{\emph{$\gamma$}}}_{ii}}

\newcommand\AZ{A_{\scalebox{0.5}{$Z$}}}
\newcommand\cd[2]{\nabla_{\!#1}{#2}}

\newcommand\cdh[2]{\nabla^{\scalebox{0.5}{\emph{h}}}_{\!#1}{#2}}
\newcommand\cdt[2]{\widetilde{\nabla}_{\!#1}{#2}}

\newcommand\hL{g_{\scalebox{0.4}{\emph{L}}}}
\newcommand\gL{g^{\scalebox{0.5}{$\gamma$}}_{\scalebox{0.4}{\emph{L}}}}

\newcommand\Ric{\text{Ric}_{\scalebox{0.6}{\emph{g}}}}

\newcommand\comma{\hspace{.2in},\hspace{.2in}}
\newcommand\commas{\hspace{.1in},\hspace{.1in}}
\newcommand\commass{\hspace{.05in},\hspace{.05in}}

\newtheorem{innercustomgeneric}{\customgenericname} 
\providecommand{\customgenericname}{}
\newcommand{\newcustomtheorem}[2]{%
  \newenvironment{#1}[1]
  {%
   \renewcommand\customgenericname{#2}%
   \renewcommand\theinnercustomgeneric{##1}%
   \innercustomgeneric
  }
  {\endinnercustomgeneric}
}

\newcustomtheorem{customthm}{Theorem}
\newcustomtheorem{customlemma}{Lemma}

\begin{document}
\title[]{Geometry via Plane wave limits}
\author[]{Amir Babak Aazami}
\address{Clark University\hfill\break\indent
Worcester, MA 01610}
\email{aaazami@clarku.edu}

\maketitle
\begin{abstract}
Utilizing the covariant formulation of Penrose's plane wave limit by Blau et~al., we construct for any semi-Riemannian metric $g$ a family of ``plane wave limits."  These limits are taken along any geodesic of $g$, yield simpler metrics of Lorentzian signature, and are isometric invariants. We show that they generalize Penrose's limit to the semi-Riemannian regime and, in certain cases, encode $g$'s tensorial geometry and its geodesic deviation.  As an application of the latter, we partially extend a well known result by Hawking \& Penrose to the semi-Riemannian regime: On any semi-Riemannian manifold, if the Ricci curvature is nonnegative along any complete geodesic without conjugate points that is ``causally independent" (in a sense we make precise), then the curvature tensor along that geodesic must vanish in all normal directions. A Morse Index Theorem is also proved for such geodesics.
\end{abstract}

\section{Introduction}
In this article we extend and adapt R. Penrose's \emph{plane wave limit} \cite{penPW} as a purely geometric tool. Penrose's limit is a famous construction in gravitational physics, by which any Lorentzian manifold is shown to admit a so-called \emph{plane wave} spacetime as a limit, this limit being taken along any lightlike geodesic.  The plane waves of Penrose's limit model gravitational radiation and have a rich history.  Interestingly, their mathematical discovery, by H. Brinkmann \cite{brinkmann}, predated their discovery within physics; see, e.g., \cite{AMS}.  The literature on plane wave limits is, by now, quite large, and much of the current interest in them is via the AdS/CFT correspondence, as plane wave limits of $\text{AdS}_{m} \times \mathbb{S}^n$ have been found that provide examples of maximally supersymmetric, quantizable backgrounds in string theory; see \cite{blau}.  For more on their properties, consult \cite{blau04,blau11}.
\vskip 6pt
The crucial step for us in utilizing Penrose's limit geometrically is to work, not with Penrose's original (and beautiful) local construction, but rather with its (equally beautiful) ``covariant formulation" discovered by \cite{blau04}. The latter extends Penrose's limit beyond local coordinates and onto the entirety of the lightlike geodesic (in Section \ref{subsec:Pen1} and Theorem \ref{prop:Blau} we review both constructions).  Our first task, therefore, is to extend \cite{blau04}'s formulation to any geodesic $\gamma$\,---\,spacelike, timelike, or lightlike\,---\,on a semi-Riemannian manifold $(M,g)$ of any signature, thereby defining a ``semi-Riemannian plane wave limit," a limit that is still a Lorentzian plane wave spacetime. Geometrically, arguably the most interesting case is when $g$ is Riemannian, and we begin, in Section \ref{sec:RPWL}, with this case. After defining a Riemannian version of our plane wave limit (Definition \ref{def:PWL}), our first result (Theorem \ref{thm:main}) details how it encodes the geometry of $g$ within its own geometry:
\begin{theorem-non}
Let $(M,g)$ be a Riemannian manifold.
\begin{enumerate}[leftmargin=*]
\item[i.] $g$ is flat $\!\!\iff\!\!$ its plane wave limits are all flat,
\item[ii.] $g$ has constant sectional curvature $\!\!\iff\!\!$ its plane wave limits are all locally conformally flat,
\item[iii.] $g$ is Ricci-flat $\!\!\iff\!\!$ its plane wave limits are all Ricci-flat,
\item[iv.] \emph{$\Ric(\gamma',\gamma')$} is constant along each geodesic $\gamma \!\!\iff\!\!$ $g$'s plane wave limits all have parallel Ricci tensors,
\item[v.] \emph{$(\nabla_{\!\gamma'}\text{Rm}_{\scalebox{0.6}{$g$}})(\cdot,\gamma',\gamma',\cdot)$} vanishes along each geodesic $\gamma \!\!\iff\!\!$ $g$'s plane wave limits are all locally symmetric,
\item[vi.] $g$ has signed Ricci curvature $\!\!\iff\!\!$ its plane wave limits all have Ricci curvatures of the same sign,
\item[vii.] $g$ is geodesically complete $\!\!\iff\!\!$ its plane wave limits are all geodesically complete.
\end{enumerate}
\end{theorem-non}

Our second result is more difficult: We show (Theorem \ref{prop:Blau}) that our Riemannian plane wave limit above can still be realized as an instance of Penrose's original (locally defined) plane wave limit:

\begin{theorem-non}
The plane wave limit of the Riemannian manifold $(M,g)$ along the unit-speed geodesic $\gamma(t)$ is the Penrose plane wave limit of the Lorentzian manifold $(I \times M,-d\tau^2+g)$ along the lightlike geodesic $\tilde{\gamma}(t) \!\defeq\! (t,\gamma(t))$. 
\end{theorem-non}

In Section \ref{sec:hereditary} we then move to the fully semi-Riemannian regime. Given that our geodesics may now be spacelike, timelike, or lightlike\,---\,and that we would like to take the limit along \emph{any} of them\,---\,there is a question of how best to define the plane wave limit in this fully general setting. We resolve this question by taking guidance from \cite{blau04}'s crucial observation  that what Penrose's original limit is actually encoding is \emph{geodesic deviation}. Viewed in this light, we present (Definition \ref{def:PWL2}) a semi-Riemannian plane wave limit\,---\,valid for all signatures and geodesics\,---\,and show (Proposition \ref{prop:conj}) that it continues to encode geodesic deviation\,---\,provided that the geodesics $\gamma$ along which the limit is taken are ``\emph{causally independent}": There exists along $\gamma$ a maximal parallel orthonormal frame $\{E_1,\dots,E_{r}\}\subseteq \gamma'^{\perp}$ such that $\text{Rm}_{\scalebox{0.6}{$g$}}(E_i,\gamma',\gamma',E_j) = 0$ whenever $E_i,E_j$ have different causal characters. (Riemannian geodesics, as well as timelike or lightlike Lorentzian geodesics, are causally independent by default.) That Definition \ref{def:PWL2} encodes the geodesic deviation of precisely such geodesics is its primary virtue.
\vskip 6pt
Indeed, for our third and final result, we use our plane wave limit to extend to the semi-Riemannian regime the focusing theorem of S.~Hawking \& Penrose \cite{HP1970} (see also \cite[Prop.~12.10,12.17]{beem}). Namely, on a Lorentzian manifold $(M,\hL)$, if $\gamma$ is a complete timelike or lightlike geodesic without conjugate points such that $\text{Ric}_{\scalebox{0.6}{$\hL$}}(\gamma',\gamma') \geq 0$, then $\text{Rm}_{\scalebox{0.6}{$\hL$}}(\cdot,\gamma',\gamma',\cdot)\big|_{\gamma'^{\perp}} = 0$. Using Jacobi tensors, this was subsequently extended to Riemannian manifolds in \cite[Theorem~2]{Eschenburg}. Crucial to these results is that the metric be positive-definite on $\gamma'^{\perp} \subseteq TM$ (or on $\gamma'^{\perp}/\gamma'$ when $\gamma$ is lightlike), for otherwise $\gamma$ may not extremize the length functional (e.g., when $\gamma$ is spacelike).  Thus no such result is known on semi-Riemannian manifolds in general, although a Morse Index Theorem (using the Maslov Index) has been established, in \cite{helfer,piccione}. Using our limit, we prove (Theorem \ref{thm:conj}) the following:
\begin{theorem-non}
Let $(M,g)$ be a semi-Riemannian $n$-manifold \emph{(}$n\geq 3$\emph{)} and $\gamma$ a complete causally independent geodesic. If $\gamma$ has no conjugate points and \emph{$\text{Ric}_{\scalebox{0.6}{$g$}}(\gamma',\gamma') \geq 0$}, then \emph{$\text{Rm}_{\scalebox{0.6}{$g$}}(\cdot,\gamma',\gamma',\cdot)\big|_{\gamma'^{\perp}} = 0$}. Furthermore, causally independent geodesic segments always have finitely many conjugate points.
\end{theorem-non}
Causal independence is crucial even in the last statement, as conjugate points on an arbitrary spacelike geodesic may accumulate; see \cite{helfer,piccione_conj}.

\section{A Riemannian plane wave limit}
\label{sec:RPWL}
Our goal in this section is fourfold:
\begin{enumerate}[leftmargin=*]
\item[1.] To introduce our extension of Penrose's plane wave limit in the special first case when $(M,g)$ is a \emph{Riemannian} manifold (Definition \ref{def:PWL}).
\item[2.] To collect together the curvature properties of Lorentzian \emph{pp-wave} metrics (Proposition \ref{prop:1}), of which plane waves are a special case. Although pp-waves are well known (see, e.g., \cite{beem,flores,blau11,globke}), we rederive their properties here for the benefit of those Riemannian geometers who may be unfamiliar with them, and because they will play a crucial role in the proofs of Theorems \ref{thm:main} and \ref{thm:conj}.
\item[3.] To show how the plane wave limit encodes much of the geometry of $(M,g)$ into its own geometry (Theorem \ref{thm:main}); e.g., we will show that $(M,g)$  has constant sectional curvature if and only if all of its plane wave limits are locally conformally flat, etc.
\item[4.] Finally, to show how the Riemannian wave plane limit can be realized as an instance of Penrose's original limit (Theorem \ref{prop:Blau}), and that, complementing \cite{Blaue_Fermi}, it can also be realized from Fermi coordinates.
\end{enumerate}

\subsection{Riemannian plane wave limit.} Let $\gamma(t)$ be a maximal unit-speed geodesic in a Riemannian $n$-manifold $(M,g)$.  At $T_{\gamma(0)}M$, choose an orthonormal frame $\{E_1,\dots,E_{n-1}\}$ orthogonal to $\gamma'(0)$ and parallel transport it along $\gamma(t)$; this choice is unique up to the action of the orthogonal group $O(n-1)$ on the subspace $\gamma'(0)^{\perp} \subseteq T_{\gamma(0)}M$.  Following \cite{blau04}, define along $\gamma(t)$ the ``wave profile" functions
\beqa
\label{eqn:geod}
\Aij(t) \defeq -\text{Rm}_{\scalebox{0.6}{\emph{g}}}(E_i,\gamma',\gamma',E_j)\Big|_{\gamma(t)} \comma i,j=1,\dots,n-1,
\eeqa
where $\text{Rm}_{\scalebox{0.6}{\emph{g}}}(a,b,c,d) \defeq g(\nabla_{\!a}\nabla_{\!b}\hspace{.01in}c - \nabla_{\!b}\nabla_{\!a}c - \nabla_{\![a,b]}c,d)$ is the Riemann curvature 4-tensor of $g$; note that $\Aij(t)=A^{\scalebox{0.6}{\emph{$\gamma$}}}_{ji}(t)$.  Thanks to \cite{blau04}, all the information needed to define $g$'s plane wave limit is contained in these $\Aij$'s:

\begin{defn}[Riemannian plane wave limit]
\label{def:PWL}
Let $(M,g)$ be a Riemannian $n$-manifold and $\gamma(t)$ a maximal geodesic with domain $I \subseteq \RR$.  The Lorentzian metric \emph{$\gL$} defined on $\RR \times I \times \RR^{n-1} \subseteq \RR^{n+1} = \{(v,t,x^1,\dots,x^{n-1})\}$ by
\beqa
\label{eqn:PWL}
\text{\emph{$\gL$}} \defeq 2dvdt + \Big(\!\sum_{i,j=1}^{n-1}\!\Aij(t)x^ix^j\Big) (dt)^2 + \sum_{i=1}^{n-1} (dx^i)^2.
\eeqa
with the $\Aij(t)$'s defined via any $g$-orthonormal frame $\{E_1,\dots,E_{n-1}\} $ parallel along $\gamma(t)$ as in \eqref{eqn:geod}, is the \emph{plane wave limit of $(M,g)$ along $\gamma$.}
\end{defn}

Although we could have defined $\gL$ more generally by $\sum_{i,j=1}^{n-1}\Aij(t)f(x^i,x^j)$ with $f\colon\RR^2\lra \RR$ a smooth function, in Theorem \ref{prop:Blau} we will make clear the virtues of our particular choice of $f$ in \eqref{eqn:PWL}.   As a first preliminary observation, let us note that the definite article ``\emph{the} plane wave limit" is justified:

\begin{lemma}
\label{lemma:same}
In Definition \ref{def:PWL}, for any other choice of orthonormal frame $\{\bar{E}_1,\dots,\bar{E}_{n-1}\}$ parallel along $\gamma(t)$, the corresponding limit metric \emph{$\bar{g}^{\scalebox{0.6}{$\gamma$}}_{\scalebox{0.4}{\emph{L}}}$} will be isometric to \emph{$\gL$}.  Also, if $\gamma(s)$ is a geodesic reparametrization of $\gamma(t)$, then $\Aij(s) = \Aij(at+b)$ for some $a,b \in \RR$.
\end{lemma}

\begin{proof}
 If each $\bar{E}_i = \sum_{j=1}^{n-1}K_{ij}E_j$, then $(K_{ij})$ is an orthogonal matrix and $(v,t,x^i) \mapsto (v,t,\sum_{j=1}^{n-1}K_{ji}x^j)$ will provide the isometry between $\gL$ and $\bar{g}^{\scalebox{0.6}{$\gamma$}}_{\scalebox{0.4}{\emph{L}}}$.  If $\gamma(s)$ is any geodesic reparametrization of $\gamma(t)$, then it must be linear.
\end{proof} 

More than that, $\gL$ is an isometric invariant of $(M,g)$, in the following sense:

\begin{lemma}\label{lemma:iso}
If $(M,g)$, $(\widetilde{M},\tilde{g})$ are Riemannian manifolds and $\varphi\colon M \lra \widetilde{M}$ is an isometry, then the plane wave limit of $(M,g)$ along the geodesic $\gamma$ is isometric to the plane wave limit of $(\widetilde{M},\tilde{g})$ along the geodesic $\tilde{\gamma} \defeq \varphi \circ \gamma$.
\end{lemma}

\begin{proof}
Each orthonormal frame $\{\widetilde{E}_1,\dots,\widetilde{E}_{n-1}\} \subseteq \tilde{\gamma}'(0)^{\perp} \subseteq T_{\tilde{\gamma}(0)}\widetilde{M}$ parallel transported along $\tilde{\gamma}(t)$ is of the form $\{d\varphi(E_1),\dots,d\varphi(E_{n-1})\}$ for some orthonormal frame $\{E_1,\dots,E_{n-1}\} \subseteq \gamma'(0)^{\perp} \subseteq T_{\gamma(0)}M$ parallel transported along $\gamma(t)$.  As $\varphi^*\text{Rm}_{\scalebox{0.6}{$\tilde{g}$}} = \text{Rm}_{\scalebox{0.6}{$g$}}$, the result now follows.
\end{proof}
 
(Penrose's original limit satisfies the same property; see, e.g., \cite{philip}.)

\subsection{Review of pp-waves.} Metrics of the form \eqref{eqn:PWL}\,---\,examples of what are called \emph{pp-waves}\,---\,can in fact be defined in a coordinate-independent manner (see \cite{globke}), namely, as Lorentzian manifolds $(M,\hL)$ admitting a parallel lightlike vector field $N$ ($N = \partial_v$ in \eqref{eqn:PWL}) and whose Riemann curvature endomorphism $R_{\scalebox{0.6}{$\hL$}}$ satisfies 
\beqa
\label{eqn:ppwave}
\text{$R_{\scalebox{0.6}{$\hL$}}(X,Y)\,\cdot = 0$ for all $X, Y \in \Gamma(N^{\perp})$.}
\eeqa
As we'll see in Proposition \ref{prop:1} below, the ``plane wave" metric \eqref{eqn:PWL} also satisfies
\beqa
\label{eqn:planewave}
\text{$\nabla^{\scalebox{0.6}{$\hL$}}_{\!X}R_{\scalebox{0.6}{$\hL$}} = 0$ for all $X \in \Gamma(N^{\perp})$}.
\eeqa
Thus the plane waves of Penrose's limit are special cases of pp-waves.  For our purposes, however, it is best to work strictly in the so called \emph{Brinkmann coordinates} of \eqref{eqn:PWL}, which all pp-waves possess at least locally.  With that said, let us now provide an example of $\gL$.  Thus, consider any Riemannian manifold $(M,g)$ of constant sectional curvature $\lambda$.  Since $\text{Rm}_{\scalebox{0.6}{$g$}} = \frac{\lambda}{2}g\,{\tiny \owedge}\,g$, it follows that $\Aij(t) = -\lambda\delta_{ij}$ along any unit-speed geodesic $\gamma(t)$, and thus its plane wave limit $\gL$ has $\sum_{i,j=1}^{n-1}\Aij(t)x^ix^j = -\lambda\sum_{i=1}^{n-1}(x^i)^2.$
As we now show, a great deal of $g$'s geometry is actually being encoded, through this polynomial, in that of $\gL$.  In what follows, $H_{ij} \defeq \frac{\partial^2 H}{\partial x^ix^j}$, etc.:

\begin{prop}
\label{prop:1}
Set $\RR^{n+1} = \{(v,t,x^1,\dots,x^{n-1})\}$ and let $H(t,x^1,\dots,x^{n-1})$ be a smooth function defined on an open subset $\mathcal{U} \subseteq \RR^n$.  The Lorentzian metric $h$ defined on $\RR \times \mathcal{U}\subseteq \RR^{n+1}$ by
\beqa
\label{eqn:pp}
h \defeq 2dvdt + H(dt)^2+\sum_{i=1}^{n-1}(dx^i)^2
\eeqa
has the following curvature properties:
\begin{enumerate}[leftmargin=*]
\item[i.] $h$ is flat $\!\!\iff\!\! H$ is linear in $x^1,\dots,x^{n-1}$,
\item[ii.] $h$ is locally conformally flat $\!\!\iff\!\! H_{ii} = H_{jj}$ and $H_{ij} = 0$ for all $i\neq j$,
\item[iii.] $h$ is Ricci-flat $\!\!\iff\!\!\Delta H = 0$, where $\Delta$ is the Euclidean Laplacian of $H$ with respect to $x^1,\dots,x^{n-1}$,
\item[iv.] $h$ is scalar-flat,
\item[v.] $h$ is locally symmetric $\!\!\iff\!\!H_{ijk} = H_{ijt} = 0$ for all $i,j,k$,
\item[vi.] $h$ has harmonic curvature tensor $\!\!\iff\!\!\partial_i(\Delta H) = 0$ for all $i$,
\item[vii.] $h$ has parallel Ricci and Schouten tensors $\!\!\iff\!\!\Delta H$ is constant.
\end{enumerate}
Finally, $h$ is geodesically complete $\!\!\iff\!\!$ the Hamiltonian system
\beqa
\label{eqn:Ham}
\ddot{x}^i = \frac{1}{2}H_i(t,x^1(t),\dots,x^{n-1}(t))\comma i=1,\dots,n-1,
\eeqa
is complete for all initial data.
\end{prop}

\begin{proof}
All of these properties are well known (see, e.g., \cite{beem,flores,blau11,globke}), though we will rederive them here for the benefit of those Riemannian geometers who may be unfamiliar with the properties of such spaces.  Let $\nabla^{\scalebox{0.5}{\emph{h}}}$ denote the Levi-Civita connection of $h$.  To begin with, the nonvanishing Christoffel symbols of $h$ are
\beqa
\label{eqn:Ch}
\cdh{\partial_{i}}{\partial_t}= \cdh{\partial_t}{\partial_{i}}= \frac{H_{i}}{2}\partial_v \comma \cdh{\partial_t}{\partial_t} = \frac{H_t}{2}\partial_v - \frac{1}{2}\sum_{i=1}^{n-1} H_i\partial_{i},
\eeqa
from which it follows that $R_{\scalebox{0.5}{\emph{h}}}(\partial_{i},\partial_{j})\partial_{k} = 0$ for all $i,j,k=x^1,\dots,x^{n-1}$, where $R_{\scalebox{0.5}{\emph{h}}}(a,b)c \defeq  \nabla^{\scalebox{0.5}{\emph{h}}}_{\!a}\nabla^{\scalebox{0.5}{\emph{h}}}_{\!b}c - \nabla^{\scalebox{0.5}{\emph{h}}}_{\!b}\nabla^{\scalebox{0.5}{\emph{h}}}_{\!a}c - \cdh{[a,b]}{c}$ is the curvature endomorphism of $h$ (cf.~\eqref{eqn:ppwave} above, with $N = \partial_v$ the parallel lightlike vector field).  Likewise,
$$
R_{\scalebox{0.5}{\emph{h}}}(\partial_{i},\partial_{j})\partial_t =  \nabla^{\scalebox{0.5}{\emph{h}}}_{\!\partial_{i}}\nabla^{\scalebox{0.5}{\emph{h}}}_{\!\partial_{j}}\partial_t - \nabla^{\scalebox{0.5}{\emph{h}}}_{\!\partial_{j}}\nabla^{\scalebox{0.5}{\emph{h}}}_{\!\partial_{i}}\partial_t = \frac{H_{ij}}{2}\partial_v - \frac{H_{ji}}{2}\partial_v = 0,
$$
so that in fact $R_{\scalebox{0.5}{\emph{h}}}(X,Y)V = 0$ for all $X,Y \in \Gamma(\partial_v^{\perp})$ and all $V \in \mathfrak{X}(M)$.  The only components remaining of the Riemann curvature 4-tensor $\text{Rm}_{\scalebox{0.5}{\emph{h}}}$ of $h$, and indeed the only generally nonvanishing ones in the coordinates \eqref{eqn:pp}, are
\beqa
\label{eqn:psec}
\text{Rm}_{\scalebox{0.5}{\emph{h}}}(\partial_i,\partial_t,\partial_t,\partial_j) = -\frac{H_{ij}}{2}\cdot
\eeqa
This, together with the fact that $h^{tt} = 0$, yields that the only nonvanishing components of the Ricci tensor $\text{Ric}_{\scalebox{0.5}{\emph{h}}}$ of $h$ and its covariant differential are
\beqa
\label{eqn:pRic}
\text{Ric}_{\scalebox{0.5}{\emph{h}}}(\partial_t,\partial_t) =  -\frac{1}{2}\Delta H  \commas (\nabla^{\scalebox{0.5}{\emph{h}}}_{\!\partial_{\alpha}}\text{Ric}_{\scalebox{0.5}{\emph{h}}})(\partial_t,\partial_t) = -\frac{1}{2}\partial_\alpha(\Delta H)\commas  \alpha = i,t,
\eeqa
where $\Delta H \defeq \sum_{i=1}^{n-1} H_{ii}$ is the Euclidean Laplacian of $H$ with  respect to $x^1,\dots,x^{n-1}$.  Because $h^{tt} = 0$, it follows at once that $h$ is scalar-flat.  Next, 
\beqa
(\nabla^{\scalebox{0.5}{\emph{h}}}_{\!\partial_k}\text{Rm}_{\scalebox{0.6}{\emph{h}}})(\partial_i,\partial_t,\partial_t,\partial_j) \!\!&=&\!\! \partial_k(\text{Rm}_{\scalebox{0.6}{\emph{h}}}(\partial_i,\partial_t,\partial_t,\partial_j)) - \text{Rm}_{\scalebox{0.6}{\emph{h}}}(\underbrace{\cdh{\partial_k}{\partial_i}}_{0},\partial_t,\partial_t,\partial_j)\nonumber\\
&&\hspace{-.05in} -\,\text{Rm}_{\scalebox{0.6}{\emph{h}}}(\partial_i,\!\underbrace{\cdh{\partial_k}{\partial_t}}_{\text{$\frac{H_k}{2}\partial_v$}},\partial_t,\partial_j) -\text{Rm}_{\scalebox{0.6}{\emph{h}}}(\partial_i,\partial_t,\!\underbrace{\cdh{\partial_k}{\partial_t}}_{\text{$\frac{H_k}{2}\partial_v$}},\partial_j)\nonumber\\
&&\hspace{1.6in}-\,\text{Rm}_{\scalebox{0.6}{\emph{h}}}(\partial_i,\partial_t,\partial_t,\!\underbrace{\cdh{\partial_k}{\partial_j}}_{0})\nonumber\\
&\overset{\eqref{eqn:psec}}{=}& -\frac{H_{kij}}{2}\cdot\label{eqn:new}
\eeqa
(In particular, $H$ will be quadratic in $x^1,\dots.x^{n-1}$ if and only if $\nabla^{\scalebox{0.5}{\emph{h}}}_{\!X}\text{Rm}_{\scalebox{0.5}{\emph{h}}} = 0$ for all $X \in \Gamma(\partial_v^{\perp})$; cf.~\eqref{eqn:planewave} above.) Replacing $\partial_k$ with $\partial_t$ yields $-\frac{H_{tij}}{2}$, with all other components of $\nabla^{\scalebox{0.5}{\emph{h}}}\text{Rm}_{\scalebox{0.5}{\emph{h}}}$ vanishing.  From the curvature properties derived thus far, it follows easily that the only nonvanishing components of $h$'s Weyl curvature tensor (assuming $n \geq 3$),
$$
W_{\scalebox{0.5}{\emph{h}}} \defeq \text{Rm}_{\scalebox{0.5}{\emph{h}}} - \frac{1}{n-1}\text{Ric}_{\scalebox{0.5}{\emph{h}}}\,{\tiny \owedge}\,h + \cancelto{0}{\frac{\text{scal}_{\scalebox{0.5}{\emph{h}}}}{2n(n-1)}h\,{\tiny \owedge}\,h},
$$
are
$$
W_{\scalebox{0.5}{\emph{h}}}(\partial_i,\partial_t,\partial_t,\partial_j) = -\frac{H_{ij}}{2} + \frac{\Delta H}{2(n-1)}\delta_{ij}.
$$
All of the properties i.-vii.~now follow easily (for vi., recall that $h$ has harmonic curvature tensor if $(\nabla^{\scalebox{0.6}{$h$}})^*\text{Rm}_{\scalebox{0.6}{$h$}} = 0$, where $(\nabla^{\scalebox{0.6}{$h$}})^*$ is the adjoint of $\nabla^{\scalebox{0.6}{$h$}}$; cf., e.g., \cite[p.~59]{petersen}).  Finally, if $\tilde{\gamma}(s) = (\tilde{\gamma}^v(s),\tilde{\gamma}^t(s),\tilde{\gamma}^1(s),\dots,\tilde{\gamma}^{n-1}(s))$ is any geodesic of $h$, then a straightforward computation yields
\beqa
\left.\begin{array}{lcl}
\ddot{\tilde{\gamma}}^v \!\!&=&\!\! -\frac{\dot{\tilde{\gamma}}^t}{2}\left(H_t\,\dot{\tilde{\gamma}}^t + \sum_{i=1}^{n-1}H_i\,\dot{\tilde{\gamma}}^i\right),\\
\ddot{\tilde{\gamma}}^t \!\!&=&\!\! 0,\label{eqn:t}\\
\ddot{\tilde{\gamma}}^i \!\!&=&\!\! \frac{(\dot{\tilde{\gamma}}^t)^2}{2}H_i\comma i=1,\dots,n-1.
\end{array}\right\}
\eeqa
As $\tilde{\gamma}^t(s)$ is linear, and as $\ddot{\tilde{\gamma}}^v$ is independent of $\tilde{\gamma}^v,\dot{\tilde{\gamma}}^v$, these equations will be determined once the ($\tilde{\gamma}^v$-independent) $\ddot{\tilde{\gamma}}^i$ equations are. Thus the completeness of the latter will determine that of the geodesic.
\end{proof}

\subsection{Geometric properties preserved by the plane wave limit.}\label{subsec:geom} Bringing together Definition \ref{def:PWL} and \eqref{eqn:psec}, let us emphasize the all-important property that the only nonvanishing components of $\text{Rm}_{\scalebox{0.6}{$\gL$}}$ are
\beqa
\label{eqn:crucial}
\text{Rm}_{\scalebox{0.6}{$\gL$}}(\partial_i,\partial_t,\partial_t,\partial_j)=-\Aij(t) = \text{Rm}_{\scalebox{0.6}{$g$}}(E_i,\gamma',\gamma',E_j).
\eeqa
Recall that constant-curvature manifolds have plane wave limits $\gL$ with $\sum_{i,j=1}^{n-1}\Aij(t)x^ix^j = -\lambda\sum_{i=1}^{n-1}(x^i)^2$.  Thanks to Proposition \ref{prop:1}, we now know that such $\gL$'s are locally conformally flat and locally symmetric\,---\,suggesting that $g$'s curvature has been encoded, via \eqref{eqn:crucial}, not just in the metric component $(\gL)_{tt}$, but in the actual \emph{tensorial geometry} of $\gL$. As our first application of the plane wave limit, we now show that this is no coincidence:

\begin{thm}
\label{thm:main}
Let $(M,g)$ be a Riemannian manifold.
\begin{enumerate}[leftmargin=*]
\item[i.] $g$ is flat $\!\!\iff\!\!$ its plane wave limits are all flat,
\item[ii.] $g$ has constant sectional curvature $\!\!\iff\!\!$ its plane wave limits are all locally conformally flat,
\item[iii.] $g$ is Ricci-flat $\!\!\iff\!\!$ its plane wave limits are all Ricci-flat,
\item[iv.] \emph{$\Ric(\gamma',\gamma')$} is constant along each geodesic $\gamma \!\!\iff\!\!$ $g$'s plane wave limits all have parallel Ricci tensors,
\item[v.] \emph{$(\nabla_{\!\gamma'}\text{Rm}_{\scalebox{0.6}{$g$}})(\cdot,\gamma',\gamma',\cdot)$} vanishes along each geodesic $\gamma \!\!\iff\!\!$ $g$'s plane wave limits are all locally symmetric,
\item[vi.] $g$ has signed Ricci curvature $\!\!\iff\!\!$ its plane wave limits all have Ricci curvatures of the same sign,
\item[vii.] $g$ is geodesically complete $\!\!\iff\!\!$ its plane wave limits are all geodesically complete.
\end{enumerate}
\end{thm}

\begin{proof}
Let $\{E_1,\dots,E_{n-1}\}$ be a parallel frame along a geodesic $\gamma(t)$ of $(M,g)$ as in \eqref{eqn:geod}, with $\Aij(t) = -\text{Rm}_{\scalebox{0.6}{\emph{g}}}(E_i,\gamma',\gamma',E_j)\big|_{\gamma(t)}$ the corresponding functions along $\gamma(t)$.  As we will be repeatedly calling upon Proposition \ref{prop:1}, recall also that the plane wave limit $\gL$ has $H_{ij} = 2\Aij(t)$.  We now proceed case-by-case:
\begin{enumerate}[leftmargin=*]
\item[i.] $g$ is flat if and only if each
$$
\text{Rm}_{\scalebox{0.6}{\emph{g}}}(E_i,\gamma',\gamma',E_i)\Big|_{\gamma(t)} = 0 \comma i=1,\dots,n-1,
$$
because at each $t$ this is the sectional curvature of the 2-plane spanned by the orthonormal pair $\{\gamma'(t),E_i\big|_{\gamma(t)}\}$, and every 2-plane can be represented in this way.  That
$$
\text{Rm}_{\scalebox{0.6}{\emph{g}}}\Big(\frac{E_i+E_j}{\sqrt{2}},\gamma',\gamma',\frac{E_i+E_j}{\sqrt{2}}\Big)\Big|_{\gamma(t)} = 0
$$
then yields 
$$
\text{Rm}_{\scalebox{0.6}{\emph{g}}}(E_i,\gamma',\gamma',E_j)\Big|_{\gamma(t)} = 0 \comma i,j=1,\dots,n-1.
$$
Thus $g$ is flat if and only if each $\Aij(t) = 0$; by Proposition \ref{prop:1}, this is equivalent to $\gL$ being flat.

\item[ii.] As we saw in Section \ref{sec:RPWL} above, $g$ has constant sectional curvature $\lambda$ if and only if the plane wave limit $\gL$ along any unit-speed geodesic $\gamma(t)$ has $\sum_{i,j=1}^{n-1}\Aij(t)x^ix^j = -\lambda\sum_{i=1}^{n-1}(x^i)^2.$ By Proposition \ref{prop:1}, such a $\gL$ is locally conformally flat (in fact also locally symmetric).  For the converse, suppose that every $\gL$ is locally conformally flat.  Then at each $p \in M$, the following holds: For any geodesic $\gamma(t)$ starting at $p$ in the direction $\gamma'(0) \defeq V$, and for any orthonormal pair $E_i,E_j \in T_pM$ orthogonal to $V$, the conditions $A^{\scalebox{0.6}{\emph{$\gamma$}}}_{ii}(t) = A^{\scalebox{0.6}{\emph{$\gamma$}}}_{jj}(t)$ and $\Aij(t) = 0$ yield, at $t=0$,
$$
\hspace{.26in}\text{Rm}_{\scalebox{0.6}{\emph{g}}}(E_i,V,V,E_i) = \text{Rm}_{\scalebox{0.6}{\emph{g}}}(E_j,V,V,E_j) \defeq \lambda_{\scalebox{0.6}{\emph{V}}} \commas \text{Rm}_{\scalebox{0.6}{\emph{g}}}(E_i,V,V,E_j) = 0,
$$
where $A^{\scalebox{0.6}{\emph{$\gamma$}}}_{ii}(0) = A^{\scalebox{0.6}{\emph{$\gamma$}}}_{jj}(0) \defeq -\lambda_{\scalebox{0.6}{\emph{V}}}$.  From this it follows that all 2-planes containing $V$ have sectional curvatures $\lambda_{\scalebox{0.6}{\emph{V}}}$.  Now let $W \in T_pM$ be any other vector, and choose a unit vector $X$ orthogonal to $V$ and $W$.  Then
$$
\lambda_{\scalebox{0.6}{\emph{V}}} = \underbrace{\,\text{Rm}_{\scalebox{0.6}{\emph{g}}}(X,V,V,X)\,}_{\text{$\text{Rm}_{\scalebox{0.6}{\emph{g}}}(V,X,X,V)$}} = \lambda_{\scalebox{0.6}{\emph{X}}} = \underbrace{\,\text{Rm}_{\scalebox{0.6}{\emph{g}}}(W,X,X,W)\,}_{\text{$\text{Rm}_{\scalebox{0.6}{\emph{g}}}(X,W,W,X)$}} = \lambda_{\scalebox{0.6}{\emph{W}}}.
$$
Thus \emph{all} 2-planes at $T_pM$ have the same sectional curvatures.  But a Riemannian manifold $(M,g)$ with the property that its \emph{pointwise} sectional curvatures are all equal must, as is well known, have constant sectional curvature globally on $M$.

\item[iii.] Next, suppose that $g$ is Ricci-flat.  Then
\beqa
\label{eqn:Harris}
\hspace{.2in}\sum_{i=1}^{n-1}\Aii(t) = -\sum_{i=1}^{n-1}\text{Rm}_{\scalebox{0.6}{\emph{g}}}(E_i,\gamma',\gamma',E_i)\Big|_{\gamma(t)} = -\text{Ric}_{\scalebox{0.6}{\emph{g}}}(\gamma',\gamma')\Big|_{\gamma(t)} = 0.
\eeqa
But by \eqref{eqn:pRic}, the only nonvanishing component of $\gL$'s Ricci tensor is $\text{Ric}_{\scalebox{0.6}{$\gL$}}(\partial_t,\partial_t) = -\!\sum_{i=1}^{n-1}\Aii(t)$, hence $\gL$ is Ricci-flat.  As for the converse, if each $\gL$ is Ricci-flat, then each $\text{Ric}_{\scalebox{0.6}{\emph{g}}}(\gamma',\gamma')\big|_{\gamma(0)} = 0$ by \eqref{eqn:Harris}.  A polarization argument then yields $\text{Ric}_{\scalebox{0.6}{\emph{g}}}(v,w) = 0$ for all $v,w \in T_{\gamma(0)}M$.

\item[iv.] The case when each $\Ric(\gamma',\gamma')$ is constant along $\gamma(t)$ is similar to iii.; indeed, if $\gamma'(\Ric(\gamma',\gamma')) = 0$, then, similarly to \eqref{eqn:Harris},
$$
-\gamma'(\Ric(\gamma',\gamma')) = \frac{d}{dt}\Big(\sum_{i=1}^{n-1}\Aii(t)\Big) = \frac{1}{2}\frac{d(\Delta H)}{dt} = 0.
$$
By Proposition \ref{prop:1}, this is equivalent to $\gL$ having parallel Ricci tensor.

\item[v.] Finally, if the symmetric 2-tensor $(\nabla_{\!\gamma'}\text{Rm}_{\scalebox{0.6}{$g$}})(\cdot,\gamma',\gamma',\cdot)$ vanishes, then
$$
(\nabla_{\!\gamma'}{\text{Rm}_{\scalebox{0.6}{\emph{g}}}})(E_i,\gamma',\gamma',E_j)\Big|_{\gamma(t)} = \gamma'({\text{Rm}_{\scalebox{0.6}{\emph{g}}}}(E_i,\gamma',\gamma',E_j)) = 0,
$$
where we note that $\nabla_{\!\gamma'}\gamma' = \nabla_{\!\gamma'}E_i = 0$.  Thus each $\Aij(t)$ is constant.  By Proposition \ref{prop:1}, this is equivalent to $\gL$ being locally symmetric.

\item[vi.] The proof is similar to that of iii.~above.

\item[vii.] By Proposition \ref{prop:1}, $\gL$ will be geodesically complete if and only if its corresponding Hamiltonian system \eqref{eqn:Ham} is complete. As the $H$ of $\gL$ is quadratic in $x^1,\dots,x^{n-1}$, its Hamiltonian system is  linear in $x^1,\dots,x^{n-1}$.  Therefore, it will be complete if and only if each $\Aij(t)$ is defined for all $t \in \RR$, which will be the case if and only if the geodesic $\gamma(t)$ of $g$ is complete.
\end{enumerate}
This completes the proof.
\end{proof}

\subsection{Relation to Penrose's original plane wave limit.}\label{subsec:Pen1} Moving on, we now shed light on the origin of $\gL$ by proving that it is, in fact, a limit in Penrose's original sense. To do so, let us first briefly remark on the latter. On the face of it, Penrose's original construction would seem to have no analogue in Riemannian geometry.  That is because, although it is a ``blowing up" process of a Lorentzian metric $\hL$, this blowing up takes place not at a point, but rather along a \emph{lightlike} geodesic $\gamma$; i.e., one for which $\hL(\gamma',\gamma') = 0$ but $\gamma' \neq 0$ (see  \cite{andersonCG}).  Such geodesics have several non-Riemannian features, but the one in particular on which Penrose's construction depends is that lightlike geodesics can be used to construct local coordinates $(x^1,\dots,x^n)$ in which $\hL$ will take the form
\beqa
\label{eqn:stillL}
(\hL)_{ij} = \begin{pmatrix}
0 & 1 & 0 & \cdots & 0\\
1 & g_{22} & g_{23} & \cdots & g_{2n}\\
0 & g_{32} & \textcolor{red}{g_{33}} & \textcolor{red}{\cdots} & \textcolor{red}{g_{3n}}\\
\vdots & \vdots & \textcolor{red}{\vdots} & \textcolor{red}{\ddots} & \textcolor{red}{\vdots}\\
0 & g_{n2} & \textcolor{red}{g_{n3}} & \textcolor{red}{\cdots} & \textcolor{red}{g_{2n}}
\end{pmatrix},
\eeqa
where the red submatrix is positive definite (see, e.g., \cite[p.~60-1]{pen}).  The matrix \eqref{eqn:stillL} has the property that its signature is determined solely by the red submatrix, regardless of the components $g_{2j}$.  Penrose capitalized on this fact by ``zooming infinitesimally close" to the lightlike geodesic that gave rise to \eqref{eqn:stillL}, in such a way as to make each $g_{2j} \to 0$, leaving behind a simpler\,---\,yet still Lorentzian, and non-flat\,---\,metric in the limit. As shown in \cite{blau04}, this beautiful construction nevertheless masks some important geometry, namely, that Penrose's limit does not, in fact, rely on the coordinates \eqref{eqn:stillL}, and that what it truly encodes is the \emph{geodesic deviation} of $\gamma(t)$, an issue to which we shall return in Section \ref{subsec:proof}.  Guided by \cite{blau04}, we now show that if one lifts geodesics $\gamma(t)$ of a Riemannian metric $g$ to lightlike geodesics $(t,\gamma(t))$ of the Lorentzian metric $-d\tau^2+g$, and then takes the plane wave limit of the latter via \cite{blau04}, then this is equivalent to both our Definition \ref{def:PWL} and Penrose's original local construction.
\begin{thm}
\label{prop:Blau}
The plane wave limit of the Riemannian manifold $(M,g)$ along the unit-speed geodesic $\gamma(t)$ is the Penrose plane wave limit of the Lorentzian manifold $(I \times M,-d\tau^2+g)$ along the lightlike geodesic $\tilde{\gamma}(t) \!\defeq\! (t,\gamma(t))$. 
\end{thm}

\begin{proof}
This proof is lengthy and proceeds in several steps. First, let us recall standard facts about product metrics (see, e.g., \cite[p.~89]{o1983}).  To begin with, a curve $\tilde{\gamma}(t)$ in $I \times M$ will be a geodesic with respect to $-d\tau^2+g$ if and only if its projections onto $I$ and $M$ are geodesics of $-d\tau^2$ and $g$, respectively.  Therefore, $\tilde{\gamma}(t)$ will be a geodesic if and only if it is of the form $(at+b,\gamma(t))$ with $\gamma(t)$ a geodesic of $(M,g)$.  So if $\gamma(t)$ has unit speed in $(M,g)$, then $\tilde{\gamma}(t) \defeq (t,\gamma(t))$ is a lightlike geodesic of $(I \times M,-d\tau^2+g)$ starting at $(0,\gamma(0))$ with $\tilde{\gamma}'(t) = \partial_\tau\big|_t+\gamma'(t)\big|_{\gamma(t)}$.  The first step in our proof is to take Penrose's plane wave limit of $(I \times M,-d\tau^2+g)$ along $\tilde{\gamma}(t)$ in the covariant manner of \cite{blau04}; we claim that doing so will bring us to precisely $\gL$ in \eqref{eqn:PWL}.  Indeed, choose any $g$-orthonormal frame $\{E_2,\dots,E_{n}\}$ parallel along $\gamma(t)$ in $(M,g)$ as in \eqref{eqn:geod} and lift it to a frame $\{\widetilde{E}_2,\dots,\widetilde{E}_{n}\}$ in $(I \times M,-d\tau^2+g)$ along the lightlike geodesic $\tilde{\gamma}(t)$ (the linearly independent set $\{\widetilde{E}_2,\dots,\widetilde{E}_{n},\tilde{\gamma}'\}$ will span $\tilde{\gamma}'^\perp \subseteq TM$; see \cite[Lemma~28, p.~142]{o1983}). Setting $\tilde{g} \defeq -d\tau^2+g$ and denoting by $\widetilde{\nabla}$ its Levi-Civita connection, note that each $\widetilde{E}_i$ is parallel along $\tilde{\gamma}(t)$: $\cdt{\tilde{\gamma}}{\widetilde{E}_i} = \cdt{\partial_{\tau}}{\widetilde{E}_i}+\cdt{\gamma'}{\widetilde{E}_i} = 0+\cd{\gamma'}{E_i} = 0$.  According to \cite{blau04}\,---\,and in perfect analogy to \eqref{eqn:PWL}\,---\,the plane wave limit of $(I \times M,-d\tau^2+g)$ along the lightlike geodesic $\tilde{\gamma}(t)$ will then be
\beqa
\label{eqn:L_Brink}
\text{$\hL^{\scalebox{0.6}{\emph{$\tilde{\gamma}$}}}$} \defeq 2dvdt + \Big(\!\sum_{i,j=2}^{n}\!A^{\scalebox{0.6}{\emph{$\tilde{\gamma}$}}}_{ij}(t)x^ix^j\Big) (dt)^2 + \sum_{i=2}^{n} (dx^i)^2,
\eeqa
where, because the Riemann curvature 4-tensor $\text{Rm}_{\scalebox{0.6}{$\tilde{g}$}}$ satisfies
$$
\text{Rm}_{\scalebox{0.6}{$\tilde{g}$}}(\partial_\tau,\cdot,\cdot,\cdot) = 0 \comma \text{Rm}_{\scalebox{0.6}{$\tilde{g}$}}(X,Y,Z,V) = \text{Rm}_{\scalebox{0.6}{$g$}}(X,Y,Z,V)
$$
for all vector fields $X,Y,Z,V$ that are lifts of vector fields on $M$, we have that
\beqa
\label{eqn:Blau_A}
A^{\scalebox{0.6}{\emph{$\tilde{\gamma}$}}}_{ij}(t) \defeq -\text{Rm}_{\scalebox{0.6}{$\tilde{g}$}}(\widetilde{E}_i,\tilde{\gamma}',\tilde{\gamma}',\widetilde{E}_j)\Big|_{\tilde{\gamma}(t)} = -\text{Rm}_{\scalebox{0.6}{$g$}}(E_i,\gamma',\gamma',E_j)\Big|_{\gamma(t)} = \Aij(t).
\eeqa
Thus the first step. Next, we will find a coordinate transformation in which $\hL^{\scalebox{0.6}{\emph{$\tilde{\gamma}$}}}$ takes on the local form of Penrose's original limit.  To do so, let us first outline Penrose's construction. On $(I \times M,-d\tau^2+g)$, begin by setting up a ``lightlike coordinate system" as in \eqref{eqn:stillL} that is ``centered" on $\tilde{\gamma}(t)$, as follows. First, recall the existence of Riemannian \emph{semigeodesic coordinates}: At each point of $M$, local coordinates $(r,x^2,\dots,x^{n})$ exist in which the integral curves $t \mapsto (t,x^2,\dots,x^{n})$ of $\partial_r$ are geodesics normal to the level sets of $r$:
\beqa
\label{eqn:sg0}
g = (dr)^2 + \sum_{i,j=2}^{n}g_{ij}(r,x^2,\dots,x^{n})dx^idx^j.
\eeqa
It follows that $\partial_r = \text{grad}_{\scalebox{.6}{$g$}}r$ and that $r$ is a local distance function (see, e.g., \cite[Proposition~6.41,~p.~182]{Lee}).  One may further choose such coordinates so that the integral curve $t \mapsto (t,0,\dots,0)$ coincides with the geodesic $\gamma(t)$. Having done so in a neighborhood $\mathcal{U} \subseteq M$, let us restrict $g$ to $\uu$ and $\tilde{g}$ to $I \times \uu$. Since
$
\tilde{g}\big|_{I\times \uu} = -(d\tau)^2+(dr)^2 + \sum_{i,j=2}^{n}g_{ij}(r,x^2,\dots,x^{n})dx^idx^j,
$
a final change of variables, $\tilde{\tau} \defeq \frac{1}{\sqrt{2}}(r+\tau), \tilde{r} \defeq \frac{1}{\sqrt{2}}(r-\tau)$, will put $\tilde{g}$ into the form \eqref{eqn:stillL}:
\beqa
\label{eqn:lor11}
    (\tilde{g}_{\alpha\beta}) = \underbrace{\,\begin{pmatrix}
        0 & 1  & 0 & 0 & \cdots & 0\\
        1 & 0& 0 & 0 & \cdots & 0\\
        0 & 0&  g_{22} & g_{23} & \cdots &  g_{2n}\\
        0 & 0 &  g_{32} & g_{33} & \cdots & g_{3n}\\
        \vdots & \vdots & \vdots & \vdots & \ddots &  \vdots\\
        0 & 0 & g_{n2} & g_{n3} & \cdots & g_{nn}
      \end{pmatrix}\,}_{(\tilde{g}_{\alpha\beta})~\text{in coordinates}~(\tilde{\tau},\tilde{r},x^2,\dots,x^n)} \commas g_{ij}\big((\tilde{\tau}+\tilde{r})/\sqrt{2},x^2,\dots,x^n\big).
\eeqa
Furthermore, our lightlike geodesic $\tilde{\gamma}(t)$, originally in the form $(t,t,0,\dots,0)$ in $I \times \uu$, now has the expression $\tilde{\gamma}(t) = (\sqrt{2}t,0,\dots,0)$ in the coordinates $(\tilde{\tau},\tilde{r},x^2,\dots,x^n)$; i.e., after the affine reparametrization $t \mapsto t/\sqrt{2}$, it is now precisely the integral curve though the origin of the lightlike vector field $\partial_{\tilde{\tau}}$. Finally, let us relabel our coordinate system from $(\tilde{\tau},\tilde{r},x^2,\dots,x^n)$ to $(x^0,x^1,x^2,\dots,x^n)$. The remainder of the proof now proceeds as follows:
\begin{enumerate}[leftmargin=*]
\item[1.] The metric \eqref{eqn:lor11} is now in a state in which to take Penrose's limit as in Penrose's original construction, in \cite{penPW}. As we show in \eqref{eqn:hpw} below, doing so will yield a plane wave metric $\hpw$ in so called \emph{Rosen coordinates}.
\item[2.] Following \cite{blau04}, we will then define a coordinate transformation putting $\hpw$ precisely in the form \eqref{eqn:PWL}, with the $\Aij(t)$'s given by \eqref{eqn:Blau_A} above. Thus $\hpw$ will be equivalent to $\gL$ in our Definition \ref{def:PWL}.
\end{enumerate}

We begin with 1.; although Penrose's construction is presented here in detail, the limit we take is in fact a special simpler case because $\tilde{g} = -d\tau^2+g$ is a time-symmetric Lorentzian metric (for the general case, see \cite{penPW,blau04,aazami}). First, define a new coordinate system $(\tilde{x}^0,\tilde{x}^1,\tilde{x}^2,\dots,\tilde{x}^n)$ via the following diffeomorphism $\varphi_{\scriptscriptstyle \vep}$,
\beqa
\label{eqn:tilde}
(x^0,x^1,x^2,\dots,x^n) \overset{\varphi_{\scriptscriptstyle \vep}}{\mapsto} (\tilde{x}^0,\tilde{x}^1,\tilde{x}^2,\dots,\tilde{x}^n) \defeq \Big(x^0,\frac{x^1}{\vep^2},\frac{x^2}{\vep},\dots,\frac{x^n}{\vep}\Big),
\eeqa
where $\vep > 0$ is a constant; let 
\beqa
\label{eqn:gg_vep}
\tilde{g}_{\scriptscriptstyle \vep} \defeq (\varphi_{\scriptscriptstyle \vep}^{-1})^*\tilde{g}
\eeqa
denote $\tilde{g}$ in these new coordinates.  Since $x^0$ is an affine parameter along the geodesic integral curve of $\partial_0$ through the origin\,---\,which is precisely our lightlike geodesic $\tilde{\gamma}(x^0)$\,---\,observe that the limit $\vep \to 0$ has the effect of ``zooming infinitesimally close" to our geodesic, pushing the remaining coordinates $x^1,\dots,x^n$ out to infinity.  This is precisely the limit that Penrose wishes to take.  However, before taking that limit, Penrose defines another Lorentzian metric, $\gpw$, in the new coordinates $(\tilde{x}^0,\tilde{x}^1,\tilde{x}^2,\dots,\tilde{x}^n)$, by
\beqa
\label{newmetric}
\gpw \defeq 2d\tilde{x}^0d\tilde{x}^1 + \sum_{i,j=2}^n h_{ij}d\tilde{x}^id\tilde{x}^j,
\eeqa
where each component $h_{ij}$ is (strategically) defined as follows:
\beqa
h_{ij}(\tilde{x}^0,\tilde{x}^1,\tilde{x}^2,\dots,\tilde{x}^n) \!\!&\defeq&\!\! \underbrace{\,g_{ij}((\tilde{x}^0+\vep^2\tilde{x}^1)/\sqrt{2},\vep\,\tilde{x}^2,\dots,\vep\,\tilde{x}^n)\,}_{\text{$={g}_{ij}((x^0+x^1)/\sqrt{2},x^2,\dots,x^n)$}}\label{newcomp2}.
\eeqa
Note that as $\vep \to 0$,
\beqa
\label{eqn:all0}
\lim_{\vep \to 0}  h_{ij} = {g}_{ij}(\tilde{x}^0/\sqrt{2},0,0,\dots,0) \comma i,j=2,\dots,n.
\eeqa
The most important feature of the metric $\gpw$ is that it is \emph{homothetic} to $\tilde{g}_{\scriptscriptstyle \vep}$; indeed, via \eqref{eqn:tilde} and \eqref{eqn:gg_vep}, 
$$
\underbrace{\,\tilde{g}_{\scriptscriptstyle \vep}(\partial_{\tilde{x}^i},\partial_{\tilde{x}^j})\,}_{\tilde{g}\left(d\varphi_{\scriptscriptstyle \vep}^{-1}(\partial_{\tilde{x}^i}),d\varphi_{\scriptscriptstyle \vep}^{-1}(\partial_{\tilde{x}^i})\right)}\hspace{-.28in}d\tilde{x}^i \otimes d\tilde{x}^j\bigg|_{(\tilde{x}^0,\tilde{x}^1,\dots,\tilde{x}^n)} = \tilde{g}(\partial_{x^i},\partial_{x^j})dx^i\otimes dx^j\bigg|_{\varphi_{\scriptscriptstyle \vep}^{-1}(\tilde{x}^0,\tilde{x}^1,\dots,\tilde{x}^n)},
$$
where $d\varphi_{\scriptscriptstyle \vep}^{-1}(\partial_{\tilde{x}^1}) = \vep^2\,\partial_{x^1}, d\varphi_{\scriptscriptstyle \vep}^{-1}(\partial_{\tilde{x}^2}) = \vep\,\partial_{x^2}$, etc., which yields
\beqa
dx^0 \otimes dx^1 \!\!&=&\!\! \vep^{2}\,d\tilde{x}^0 \otimes d\tilde{x}^1, \nonumber\\
g_{ij}((x^0+x^1)/\sqrt{2},x^2,\dots,x^n)\,dx^i \otimes dx^j \!\!&\overset{\eqref{newcomp2}}{=}&\!\!\vep^{2}\,h_{ij}d\tilde{x}^i \otimes d\tilde{x}^j,\nonumber
\eeqa
so that, overall,
$$
\gpw = \frac{\tilde{g}_{\scriptscriptstyle \vep}}{\vep^2}\cdot
$$
(In particular, the Levi-Civita connections of $h_{\scriptscriptstyle \vep}$ and $\tilde{g}_{\scriptscriptstyle \vep}$ are equal: $\nabla^{{\scalebox{0.55}{\text{$h_\vep$}}}} = \widetilde{\nabla}$.) Now Penrose takes the limit $\vep \to 0$\,---\,not of $\tilde{g}_{\scriptscriptstyle \vep}$, but rather of $\gpw$:
\beqa
\label{eqn:star}
\lim_{\vep \to 0} \frac{\tilde{g}_{\scriptscriptstyle \vep}}{\vep^{2}} = \lim_{\vep \to 0} \gpw.
\eeqa
Recalling \eqref{newmetric} and \eqref{eqn:all0}, this limit exists and yields the following Lorentzian metric:
$$
\hpw \defeq \lim_{\vep \to 0} \gpw = 2d\tilde{x}^0d\tilde{x}^1 + \sum_{i,j=2}^n g_{ij}(\tilde{x}^0/\sqrt{2},0,\dots,0)d\tilde{x}^id\tilde{x}^j.
$$
As it turns out, this is precisely a plane wave metric in the Rosen coordinates mentioned above. Scaling by $\tilde{x}^0 \mapsto \sqrt{2}\tilde{x}^0$ and $\tilde{x}^1 \mapsto \tilde{x}^1/\sqrt{2}$ puts Penrose's plane wave limit $\hpw$ into the final form
\beqa
\label{eqn:hpw}
\hpw = 2d\tilde{x}^0d\tilde{x}^1+\sum_{i,j=2}^{n}g_{ij}(\tilde{x}^0,0,\dots,0)d\tilde{x}^id\tilde{x}^j.
\eeqa
This takes care of 1.~above; we now finish the proof by showing that $\hpw$ is isometric to $\gL$ in Definition \ref{def:PWL}, with the $\Aij(t)$'s given by \eqref{eqn:Blau_A}. To begin with, we may express the vector fields $\{\widetilde{E}_2,\dots,\widetilde{E}_n\}$ along $\tilde{\gamma}(\tilde{x}^0)$ in the coordinate frame $\{\partial_{\tilde{0}},\partial_{\tilde{1}},\partial_{\tilde{2}},\dots,\partial_{\tilde{n}}\}$ (where each $\partial_{\,\tilde{j}} \defeq \partial/\partial \tilde{x}^j$) as
\beqa
\label{eqn:**}
\widetilde{E}_i = p_i(\tilde{x}^0) \partial_{\tilde{0}} + \sum_{k=2}^nf^k_i(\tilde{x}^0)\partial_{\tilde{k}}\,\Big|_{\tilde{\gamma}(\tilde{x}^0)} \comma i = 2,\dots,n,
\eeqa
for some smooth functions $p_i(\tilde{x}^0),f_i^k(\tilde{x}^0)$. In fact the $(n-1)\times (n-1)$ matrix of functions $(f^k_i)$ is invertible, because the $n-1$ vector fields $\widetilde{E}_i-p_i(\tilde{x}^0)\partial_{\tilde{0}}$ are linearly independent:
\beqa
\label{eqn:fzero}
\underbrace{\,\sum_{k=2}^n a^i\big(\widetilde{E}_i-p_i(\tilde{x}^0)\partial_{\tilde{0}}\big) = 0 \iff \begin{pmatrix} f^2_2& \cdots& f^2_n\\f^3_2&\cdots& f^3_n\\ \vdots&\ddots&\vdots\\f^2_n&\cdots&f^n_n \end{pmatrix}\begin{pmatrix}a^2\\a^3\\\vdots\\a^n\end{pmatrix} = \begin{pmatrix}0\\0\\\vdots\\0\end{pmatrix}\,}_{\text{and the former occurs $\iff a^2 = \cdots = a^n = 0$.}}
\eeqa 
(This is a point to which we shall return below.) Moving on, as $\tilde{\gamma}(\tilde{x}^0) = (\tilde{x}^0,0,\dots,0)$ and recalling each component $g_{ij}(\tilde{x}^0+\vep^2\tilde{x}^1/2,\vep\,\tilde{x}^2,\dots,\vep\,\tilde{x}^n)$ in \eqref{eqn:lor11} (and our final rescaling $\tilde{x}^0 \mapsto \sqrt{2}\tilde{x}^0$ and $\tilde{x}^1 \mapsto \tilde{x}^1/\sqrt{2}$ above), observe that, along $\tilde{\gamma}(\tilde{x}^0) = (\tilde{x}^0,0,\dots,0)$,
\beqa
\label{eqn:h=g}
\tilde{g}_{\scriptscriptstyle \vep} \!\!&=&\!\! \vep^2\Big(2d\tilde{x}^0d\tilde{x}^1+\sum_{i,j=2}^{n}g_{ij}(\tilde{x}^0+\vep^2\tilde{x}^1/2,\vep\,\tilde{x}^2,\dots,\vep\,\tilde{x}^n)d\tilde{x}^id\tilde{x}^j\Big)\nonumber\\
&\Rightarrow&\!\! \tilde{g}_{\scriptscriptstyle \vep}\big|_{\tilde{\gamma}} = \vep^2\hpw\big|_{\tilde{\gamma}}.
\eeqa
Thus $\hpw(\widetilde{E}_i,\widetilde{E}_j)\big|_{\tilde{\gamma}} = \vep^{-2}\tilde{g}_{\scriptscriptstyle \vep}(\widetilde{E}_i,\widetilde{E}_j)\big|_{\tilde{\gamma}} = \vep^{-2}\delta_{ij}$, while the absence of a $\partial_1$-component in \eqref{eqn:**} is due to the fact that
$$
\hpw(\widetilde{E}_i,\partial_{\tilde{0}})\big|_{\tilde{\gamma}} = \hpw(\widetilde{E}_i,\tilde{\gamma}')\big|_{\tilde{\gamma}} = \vep^{-2}\tilde{g}_{\scriptscriptstyle \vep}(\widetilde{E}_i,\tilde{\gamma}')\big|_{\tilde{\gamma}} = 0.
$$
A further consequence of \eqref{eqn:h=g}\,---\,indeed, more generally of the definitions of $\tilde{g}_{\scriptscriptstyle \vep}, h_{\scriptscriptstyle \vep}$, and $\hpw$\,---\,is that each $\widetilde{E}_i$, which was originally $\widetilde{\nabla}$-parallel along $\tilde{\gamma}(\tilde{x}^0) = (\tilde{x}^0,0,\dots,0)$, remains $\nabla^{{\scalebox{0.55}{\text{$\hpw$}}}}$-parallel along it as well. Indeed, it is straightforward to verify that
$\nabla^{{\scalebox{0.55}{\text{$\hpw$}}}}_{\!\partial_{\tilde{0}}}{\partial_{\tilde{0}}} = 0 = \cdt{\partial_{\tilde{0}}}{\partial_{\tilde{0}}}$,
while
$$
\nabla^{{\scalebox{0.55}{\text{$\hpw$}}}}_{\!\partial_{\tilde{0}}}{\partial_{\tilde{k}}}\Big|_{\tilde{\gamma}} = \cdt{\partial_{\tilde{0}}}{\partial_{\tilde{k}}}\Big|_{\tilde{\gamma}} = \frac{1}{2}\sum_{\lambda,l=2}^ng^{\lambda l}\dot{g}_{lk}\partial_\lambda \comma \dot{g}_{lk}\defeq \partial_{\tilde{0}}(g_{lk}(\tilde{x}^0,0,\dots,0)).
$$
Applying these to \eqref{eqn:**} yields
$
\nabla^{{\scalebox{0.55}{\text{$\hpw$}}}}_{\!\tilde{\gamma}'}{\widetilde{E}_i} = \nabla^{{\scalebox{0.55}{\text{$\hpw$}}}}_{\!\partial_{\tilde{0}}}{\widetilde{E}_i}\Big|_{\tilde{\gamma}'} = \cdt{\partial_{\tilde{0}}}{\widetilde{E}_i}\Big|_{\tilde{\gamma}'} = \cdt{\tilde{\gamma}'}{\widetilde{E}_i} = 0,
$
so that
\beqa
\label{eqn:f_terms}
0=\nabla^{{\scalebox{0.55}{\text{$\hpw$}}}}_{\!\tilde{\gamma}'}{\widetilde{E}_i} = \dot{p}_i(\tilde{x}^0)\partial_{\tilde{0}} + \sum_{k=2}^n \big(\dot{f}_i^k(\tilde{x}^0)\partial_{\tilde{k}} + f_i^k(\tilde{x}^0)\nabla^{{\scalebox{0.55}{\text{$\hpw$}}}}_{\!\partial_{\tilde{0}}}{\partial_{\tilde{k}}}\big)\Big|_{\tilde{\gamma}}.
\eeqa
For each $\lambda = 2,\dots,n$, the $\partial_\lambda$-term in \eqref{eqn:f_terms} thus yields the identity
\beqa
\dot{f}_i^\lambda + \frac{1}{2}\sum_{l,k=2}^nf_i^kg^{\lambda l}\dot{g}_{lk} = 0 \!\!&\overset{\times g_{\lambda s}}{\Rightarrow}&\!\! -2\sum_{\lambda =2}^n\dot{f}_i^\lambda g_{\lambda s}\Big|_{\tilde{\gamma}} = \sum_{k=2}^nf_i^k\dot{g}_{sk}\Big|_{\tilde{\gamma}}\label{eqn:f0}\\
&\overset{\times f^s_j}{\Rightarrow}&\!\! -2\sum_{s,\lambda =2}^nf^s_j\dot{f}_i^\lambda g_{\lambda s}\Big|_{\tilde{\gamma}} = \sum_{s,k=2}^nf^s_jf_i^k\dot{g}_{sk}\Big|_{\tilde{\gamma}}.\label{eqn:f1}
\eeqa
And this in turn yields the identity
\beqa
0 \!\!&=&\!\!\partial_{\tilde{0}}(\hpw(\widetilde{E}_i,\widetilde{E}_j))\big|_{\tilde{\gamma}}\nonumber\\
&=&\!\!\partial_{\tilde{0}}(\vep^{-2}\tilde{g}_{\scriptscriptstyle \vep}(\widetilde{E}_i,\widetilde{E}_j))\big|_{\tilde{\gamma}}\nonumber\\
&\overset{\eqref{eqn:**}}{=}&\!\! \sum_{k,l=2}^n\partial_{\tilde{0}}(f_i^kf_j^lg_{kl})\Big|_{\tilde{\gamma}}\nonumber\\
&=&\!\! \sum_{k,l=2}^n (\dot{f}_i^kf_j^lg_{kl} + f_i^k\dot{f}_j^lg_{kl} + f_i^kf_j^l\dot{g}_{kl})\Big|_{\tilde{\gamma}}\nonumber\\
&\overset{\eqref{eqn:f1}}{=}&\!\! \sum_{k,l=2}^n (\dot{f}_i^kf_j^lg_{kl} + f_i^k\dot{f}_j^lg_{kl} -2f^k_j\dot{f}_i^l g_{kl})\Big|_{\tilde{\gamma}}\nonumber\\
&\Rightarrow&\!\! \sum_{k,l=2}^n \dot{f}_i^k f^l_jg_{kl}\Big|_{\tilde{\gamma}} = \sum_{k,l=2}^n f_i^k\dot{f}_j^lg_{kl}\Big|_{\tilde{\gamma}}.\label{eqn:f_id}
\eeqa
We are now in a position to define the isometry between Penrose's plane wave limit $\hpw$ and our Riemannian plane wave limit $\gL$. It involves the same coordinate transformation between Rosen coordinates \eqref{eqn:hpw} and Brinkmann coordinates \eqref{eqn:pp} as defined in \cite{blau04}, which transformation is defined via the functions $f^k_i$ in \eqref{eqn:**} as follows:
$$
\def\arraystretch{1.4}
\left\{\begin{array}{lcl}
\tilde{x}^0 &\defeq& t,\\
\tilde{x}^1 &\defeq& v-\frac{1}{2}\sum_{i,j,k,l=2}^n g_{kl}(t,0,\dots,0)\dot{f}_i^k(t)f^l_j(t)x^ix^j,\\
\tilde{x}^i &\defeq& \sum_{k=2}^nf_k^i(t)x^k.
\end{array}\right.
$$
Recall \eqref{eqn:fzero} and our discussion above it: That the $(n-1)\times (n-1)$ matrix of functions $(f^k_i)$ is invertible implies that the Jacobian matrix between $(v,t,x^2,\dots,x^n)$ and $(\tilde{x}^0,\tilde{x}^1,\tilde{x}^2,\dots,\tilde{x}^n)$ is invertible, so that these coordinates are well defined; note also that $\tilde{\gamma}(t) = (0,t,0,\dots,0)$ in these coordinates.
Now, if we substitute $(v,t,x^2,\dots,x^n)$ into \eqref{eqn:hpw}, and avail ourselves of the identities \eqref{eqn:f_id} and $\hpw(\widetilde{E}_i,\widetilde{E}_j) = \sum_{k,l=2}^n f_i^kf_j^lg_{kl} = \vep^{-2}\delta_{ij}$, we will obtain, after some simplification and a final rescaling $v\mapsto \vep v,t\mapsto \vep^{-1}t,x^i\mapsto \vep x^i$,
$$
\hpw = 2dvdt +\Big(\!\!-\!\!\!\!\!\!\sum_{i,j,k,l=2}^n\!\!\!\!\big(\dot{g}_{kl}\dot{f}^k_if^l_j+g_{kl}\ddot{f}^k_if^l_j\big)x^ix^j\Big)(dt)^2 + \sum_{i=2}^n (dx^i)^2.
$$
To complete the proof, all that remains to show, therefore, is that
\beqa
\label{eqn:finalAij}
\Aij(t) = -\!\!\sum_{k,l=2}^n \!\!\big(\dot{g}_{kl}\dot{f}^k_if^l_j+g_{kl}\ddot{f}^k_if^l_j\big),
\eeqa
with $\Aij(t)$ given by \eqref{eqn:Blau_A}. Indeed, recalling that $\tilde{\gamma}'(\tilde{x}^0) = \partial_{\tilde{0}}\big|_{\tilde{\gamma}(\tilde{x}^0)}$ and $\cdt{\partial_{\tilde{0}}}{\partial_{\tilde{k}}} = \frac{1}{2}\sum_{\lambda,l=2}^ng^{\lambda l}\dot{g}_{lk}\partial_\lambda$ (with each $g_{lk}(\tilde{x}^0+\vep^2\tilde{x}^1/2,\vep\,\tilde{x}^2,\dots,\vep\,\tilde{x}^n)$), we proceed as follows:
\beqa
A^{\scalebox{0.6}{\emph{$\tilde{\gamma}$}}}_{ij}(\tilde{x}^0) \!\!&=&\!\! -\text{Rm}_{\scalebox{0.6}{$\tilde{g}_{\scriptscriptstyle \vep}$}}(\widetilde{E}_i,\tilde{\gamma}',\tilde{\gamma}',\widetilde{E}_j)\Big|_{\tilde{\gamma}(\tilde{x}^0)}\nonumber\\
&=&\!\! -\!\sum_{k,l=2}^n f_i^kf_j^l\underbrace{\,\text{Rm}_{\scalebox{0.6}{$\tilde{g}_{\scriptscriptstyle \vep}$}}(\partial_{\tilde{k}},\partial_{\tilde{0}},\partial_{\tilde{0}},\partial_{\tilde{l}})\,}_{\text{$\cdt{\partial_{\tilde{k}}}{\!\!\cancelto{0}{\cdt{\partial_{\tilde{0}}}{\partial_{\tilde{0}}}}} - \cdt{\partial_{\tilde{0}}}{\!\!\cdt{\partial_{\tilde{k}}}{\partial_{\tilde{0}}}}$}}\Big|_{\tilde{\gamma}(\tilde{x}^0)}\nonumber\\
&=&\!\! \sum_{k,l=2}^n f_i^kf_j^l \Big[\underbrace{\,\partial_{\tilde{0}}\big(\tilde{g}_{\scriptscriptstyle \vep}(\cdt{\partial_{\tilde{k}}}{\partial_{\tilde{0}}},\partial_{\tilde{l}})\big)\,}_{\text{$\frac{1}{2}\ddot{g}_{lk}$}} - \underbrace{\,\tilde{g}_{\scriptscriptstyle \vep}(\cdt{\partial_{\tilde{k}}}{\partial_{\tilde{0}}},\cdt{\partial_{\tilde{0}}}{\partial_{\tilde{l}}})\,}_{\text{$\frac{1}{4}\sum_{s,p=2}^n g^{sp}\dot{g}_{sk}\dot{g}_{pl}$}}\Big]\Big|_{\tilde{\gamma}(\tilde{x}^0)}\nonumber\\
&\overset{\eqref{eqn:f0}}{=}&\!\! \sum_{k,l=2}^n \Big[\frac{1}{2}f_i^kf_j^l\ddot{g}_{kl} + \frac{1}{2}\dot{f}_i^kf_j^l\dot{g}_{kl}\Big]\Big|_{\tilde{\gamma}(\tilde{x}^0)}\nonumber\\
&\overset{\partial_{\tilde{0}}\eqref{eqn:f1}}{=}&\!\! \frac{1}{2}\!\sum_{k,l=2}^n \Big[f_j^l(-2\ddot{f}_i^kg_{kl}-3\dot{f}_i^k\dot{g}_{kl}) + \dot{f}_i^kf_j^l\dot{g}_{kl}\Big]\Big|_{\tilde{\gamma}(\tilde{x}^0)}\nonumber\\
&=&\!\! -\!\!\sum_{k,l=2}^n \!\!\big(\dot{g}_{kl}\dot{f}^k_if^l_j+g_{kl}\ddot{f}^k_if^l_j\big)\Big|_{\tilde{\gamma}(\tilde{x}^0)}.\nonumber
\eeqa
Since $A^{\scalebox{0.6}{\emph{$\tilde{\gamma}$}}}_{ij}(\tilde{x}^0) = \Aij(\tilde{x}^0)$ by \eqref{eqn:Blau_A}, since $\tilde{x}^0 = t$, and finally because $\tilde{\gamma}(\tilde{x}^0) = (0,\tilde{x}^0,0,\dots,0) = (0,t,0,\dots,0) = \tilde{\gamma}(t)$, this confirms \eqref{eqn:finalAij} and completes the proof.
\end{proof}

Note that \eqref{eqn:star} now justifies the terminology ``plane wave \emph{limit}" for $\gL$.

\subsection{Relation to Fermi coordinates.}\label{subsec:Fermi} It is also worthwhile discussing the local nature of the Riemannian plane wave limit $\gL$ in another sense, complementing \cite{Blaue_Fermi}. In particular, we now show that $\gL$ can also be obtained locally from \emph{Fermi} coordinates along $\gamma(t)$.  Thus, let $\{E_1,\dots,E_{n-1}\}$ be an orthonormal frame along $\gamma(t)$ as in \eqref{eqn:geod}.  Then it is well known that local ``Fermi" coordinates $(t,x^1,\dots,x^{n-1})$ can be found in a neighborhood $\mathcal{U}$ satisfying the following properties:
\begin{enumerate}[leftmargin=*]
\item[i.] The portion of $\gamma(t)$ within $\mathcal{U}$, assumed to comprise an embedded 1-submanifold, has the expression $\gamma(t) = (t,0,\dots,0)$, with $\partial_t\big|_{\gamma(t)} = \gamma'(t)$ and $\partial_i\big|_{\gamma(t)} = E_i\big|_{\gamma(t)}$, so that $g\big|_{\gamma(t) \cap \mathcal{U}} = \text{diag}(1,1,\dots,1)$.

\item[ii.] For every $\gamma(t_0) \in \mathcal{U}$ and every $v\defeq \sum_{i=1}^{n-1}v^iE_i \in T_{\gamma(t_0)}M$ orthogonal to $\gamma'(t_0)$, the geodesic $\gamma_v(s)$ starting at $\gamma(t_0)$ in the direction $v$ has the coordinate expression $\gamma_v(s) = (t_0,sv^1,\dots,sv^{n-1})$ (see, e.g., \cite[Proposition~5.26,~p.~136]{Lee}; from this it follows that $\Gamma^\alpha_{ij}\big|_{\gamma(t)} = 0$ and that $\partial_kg_{ij}\big|_{\gamma(t)} = 0$ for all $i,j,k=1,\dots,n-1$ and $\alpha = t,1,\dots,n-1$).

\item[iii.] Finally, using that $J(s) \defeq \partial_t\big|_{\gamma_v(s)}$ is a Jacobi field along $\gamma_v(s)$ satisfying $J'(0) = \sum_{i=1}^{n-1}v^i\cd{\partial_i}{\partial_t}\big|_{\gamma_v(0)}\! =\! \sum_{i=1}^{n-1}v^i\cd{\partial_t}{\partial_i}\big|_{\gamma(t_0)}\! =\! \sum_{i=1}^{n-1}v^i\cd{\gamma'(t_0)}{E_i} = 0$, one can show that the metric component $g_{tt}$ Taylor expands as
\beqa
\hspace{.25in}g_{tt}\big|_{(t,x^1,\dots,x^{n-1})} \!\!&=&\!\! 1 - \sum_{i,j=1}^{n-1}(\text{Rm}_{\scalebox{0.6}{$g$}})_{ittj}\Big|_{(t,0,\dots,0)}x^ix^j + \mathcal{O}(|x|^3).\label{eqn:gtt0}
\eeqa
(Alternatively, see \cite[p.~186ff.]{gray}.) Since $\partial_t|_{(t,0,\dots,0)} = \gamma'(t)$ and $\partial_i\big|_{(t,0,\dots,0)} = E_i\big|_{\gamma(t)}$, \eqref{eqn:gtt0} resembles its counterpart in \eqref{eqn:PWL}:
\beqa
\label{eqn:gtt}
g_{tt}\big|_{(t,x^1,\dots,x^{n-1})} = 1 + \sum_{i,j=1}^{n-1}\Aij(t)x^ix^j + \mathcal{O}(|x|^3).
\eeqa
\end{enumerate}
This is more than just a resemblance.  Indeed, let us now form the Lorentzian manifold $(I \times M,\hL\defeq -d\tau^2+g)$, lift $\gamma(t)$ to the lightlike geodesic $\tilde{\gamma}(t) \defeq (t,\gamma(t))$ as in Theorem \ref{prop:Blau}, and finally lift the Fermi coordinates $(t,x^1,\dots,x^{n-1})$ to $(\tau,t,x^1,\dots,x^{n-1})$ on $I \times M$.  Observe that with respect to these coordinates, $\tilde{\gamma}(t) = (t,t,0,\dots,0)$ and $\hL\big|_{(t,t,0,\dots,0)} = \text{diag}(-1,1,\dots,1)$.
To avoid confusion in what follows, let us relabel the affine parameter as $\tilde{\gamma}(s) = (s,s,0,\dots,0)$.  To arrive at \eqref{eqn:PWL}, define new coordinates $(v,u,x^1,\dots,x^{n-1})$ by setting
$
v \defeq \frac{1}{2}(t-\tau), u \defeq \frac{1}{2}(t+\tau).
$
In these coordinates, $\tilde{\gamma}(s) = (0,s,0,\dots,0)$ is an integral curve of $\partial_u$.  Setting $s=u$, we now extend $\hL$ to the entire domain of $(v,u,x^1,\dots,x^{n-1})$ by Taylor expanding the components of $g$ as above\,---\,but with the following stipulation: \emph{We will Taylor expand only the component $g_{uu}$, leaving unchanged all other components from the values they had along $\tilde{\gamma}(s)$.}  (By expanding only in directions parallel to $\tilde{\gamma}(u)$; i.e., only along the integral curves of $\partial_u$, we are mimicing Penrose's construction of ``zooming infinitesimally close" to the lightlike geodesic $\tilde{\gamma}(s)$.) As $\partial_u = \frac{1}{2}(\partial_t+\partial_\tau)$ and $\partial_v = \frac{1}{2}(\partial_t-\partial_\tau)$, we thus have $(\hL)_{vv} = (\hL)_{vj} = (\hL)_{uj} = 0, (\hL)_{vu} = \frac{1}{2}$, while
$$
(\hL)_{uu} = \frac{1}{4}\hL(\partial_t+\partial_\tau,\partial_t+\partial_\tau) \overset{\eqref{eqn:gtt}}{=} \frac{1}{4}\sum_{i,j=1}^{n-1}\Aij\Big|_{(0,u,0,\dots,0)}x^ix^j + \mathcal{O}(|x|^3).
$$
Thus with our stipulation in place, $\hL$ locally takes the form
$$
\hL\big|_{(v,u,x^1,\dots,x^{n-1})} = dvdu + \frac{1}{4}\Big(\!\sum_{i,j=1}^{n-1}\!\Aij(u)x^ix^j\Big) (dt)^2 + \sum_{i=1}^{n-1} (dx^i)^2.
$$
After scaling via $u \mapsto 2t$, this is precisely the plane wave limit metric $\gL$ in \eqref{eqn:PWL} (recall Lemma \ref{lemma:same}).  This shows that Fermi coordinates along the geodesic $\gamma(t)$ will yield $\gL$ in \emph{Brinkmann coordinates} \eqref{eqn:pp}.  Now, if we had instead used Fermi coordinates with respect to a hypersurface \emph{orthogonal} to $\gamma(t)$, then locally $g = (dt)^2 + \sum_{i,j=1}^{n-1}g_{ij}(t,x^1,\dots,x^{n-1})dx^idx^j$, with $\gamma(t) = (t,0,\dots,0)$ now an integral curve of $\partial_t$ (see \cite[p.~183]{Lee}; in the case of a hypersurface, the orthogonality relation is preserved even off of it).  If we apply Penrose's original scaling argument to $-d\tau^2+g$, then we would have arrived at $\gL$ in so called \emph{Rosen coordinates}; see \cite{aazami}.  Of course, the virtue of Definition \ref{def:PWL}, which follows \cite{blau04}, is that it does not rely on any local coordinates of $(M,g)$.

\subsection{Plane wave limit of vector fields.}\label{subsec:sanchez} To close this section, we briefly present a generalization of the Riemannian plane wave limit. Thus, let us suppose that our unit-speed geodesic $\gamma(t)$ is in fact an integral curve of a smooth unit-length vector field $Z$ with geodesic flow: $\cd{Z}{Z} = 0$.  Let $Z^{\perp} \subseteq TM$ denote its orthogonal complement, which may or may not be integrable, and define the linear endomorphism
$$
\AZ\colon Z^{\perp} \lra Z^{\perp} \comma X \mapsto \AZ(X) \defeq -\cd{X}{Z}.
$$
Now let $\{E_1,\dots,E_{n-1}\}$ be a local orthonormal frame orthogonal to $Z$ and parallel along its integral curves: $\cd{Z}{E_i} = 0$ for each $i=1,\dots,n-1$ (cf., e.g., \cite[p.~237]{petersen}).  Relative to this frame, $\AZ$ has (generally non-symmetric) matrix entries $(\AZ)_{ij} = -g(\cd{E_j}{Z},E_i)$; their derivatives along $Z$ satisfy
\beqa
Z(\AZ)_{ij} &=&\!\! -g(\cd{Z}{\cd{E_j}{Z}},E_i) - g(\cd{E_j}{Z},\cancelto{0}{\cd{Z}{E_i}})\nonumber\\
&=&\!\! -\text{Rm}_{\scalebox{0.6}{$g$}}(Z,E_j,Z,E_i) - g(\cd{E_j}{\cancelto{0}{\cd{Z}{Z}}},E_i) - g(\cd{[Z,E_j]}{Z},E_i).\nonumber
\eeqa 
As $[Z,E_j] = \sum_{k=1}^{n-1}(\AZ)_{kj}E_k$, and as $\text{Rm}_{\scalebox{0.6}{$g$}}(Z,E_j,Z,E_i)\big|_{\gamma(t)} = \Aij(t)$, we thus arrive at the following Bochner-type formula relating $\Aij(t)$ and $\AZ$:
\beqa
\label{eqn:AZ2}
\Aij(t) = -\frac{d(\AZ)_{ij}}{dt} + \sum_{k=1}^{n-1}(\AZ)_{ik}(\AZ)_{kj} \Big|_{\gamma(t)} = -\frac{d(\AZ)_{ij}}{dt} + (\AZ^2)_{ij}\Big|_{\gamma(t)}.
\eeqa
Thus, if $Z$ and $\{E_1,\dots,E_{n-1}\}$ as above are present, then \eqref{eqn:AZ2} directly relates the curvature of the plane wave limit $\gL$ to geometric properties of the flow of $Z$, namely, its divergence (the trace of $\AZ$), its ``twist" (the anti-symmetric part of $\AZ$), and its ``shear" (the trace-free symmetric part of $\AZ$).

\section{A semi-Riemannian plane wave limit}
\label{sec:hereditary}

\subsection{semi-Riemannian plane wave limit.} We now generalize to the case of a semi-Riemannian $n$-manifold $(M,g)$ of any signature. Let $\gamma(t)$ be a maximal geodesic in $(M,g)$ of any causal character:
\beqa
\left\{\begin{array}{ccc}
\text{``spacelike"} &\text{if}& \text{$g(\gamma',\gamma') > 0$},\\
\text{``timelike"} &\text{if}& \text{$g(\gamma',\gamma') < 0$},\\
\text{``lightlike"} &\text{if}& \text{$g(\gamma',\gamma') = 0$}.
\end{array}\right.\nonumber
\eeqa
Whatever the index of $g$, each orthogonal complement $\gamma'(t)^{\perp} \subseteq T_{\gamma(t)}M$ always has dimension $n-1$ (see \cite[Lemma~22,~p.~49]{o1983}). Depending on the causality of the geodesic, however, the following is true:

\begin{lemma}
\label{lemma:r}
Let $\gamma(t)$ be a geodesic of a semi-Riemannian manifold $(M,g)$ with index $\nu$. At $(T_{\gamma(0)}M,g|_{\gamma(0)})$, choose an orthonormal frame $\{E_1,\dots,E_{r}\} \subseteq T_{\gamma(0)}M$ orthogonal to $\gamma'(0)$ and parallel transport it along $\gamma(t)$, where $r \leq n-1$ is the maximal number of orthonormal vectors in $\gamma'(0)^{\perp}$ possible. Then the following is true:
\begin{enumerate}[leftmargin=*]
\item[1.] If $\gamma'(0)$ is spacelike, then $r=n-1$ and $\{E_1,\dots,E_{r}\}$ has index $\nu$.
\item[2.] If $\gamma'(0)$ is timelike, then $r=n-1$ and $\{E_1,\dots,E_{r}\}$ has index $\nu-1$.
\item[3.] If $\gamma'(0)$ is lightlike, then $r=n-2$ and $\{E_1,\dots,E_{r}\}$ has index $\nu-1$.
\end{enumerate}
\end{lemma}

\begin{proof}
If $\gamma'(0)$ is spacelike or timelike, then $W \defeq \text{span}\{\gamma'(0)\} \subseteq T_{\gamma(0)}M$ is nondegenerate, hence so is $W^{\perp}$ and $T_{\gamma(0)}M = W \oplus W^\perp$, so that the indices of $W$ and $W^{\perp}$ sum to $\nu$ (see \cite[Lemmas~23~\&~26,~p.~49~\&~51]{o1983}). This proves 1.~and 2. If $\gamma'(0)$ is lightlike, then $W$ and $W^\perp$ are both degenerate. With respect to any orthogonal basis $\{e_1,\dots,e_\nu,e_{\nu+1},\dots,e_n\} \subseteq T_{\gamma(0)}M$ with $e_1,\dots,e_\nu$ timelike and $\gamma'(0) = e_1+e_{\nu+1}$, we may represent $\gamma'(0)^{\perp}$ as $\text{span}\{e_1+e_{\nu+1},e_2,\dots,e_\nu,e_{\nu+2},\dots,e_n \}$. This proves 3.
\end{proof}

With that established, we may once again define along $\gamma(t)$ the ``wave profile" functions \eqref{eqn:geod}, but this time, with Lemma \ref{lemma:r} in mind, we do so as follows,
\beqa
\label{eqn:geod2}
\Aij(t) \defeq -\vep_i|\vep_j|\,\text{Rm}_{\scalebox{0.6}{\emph{g}}}(E_i,\gamma',\gamma',E_j)\Big|_{\gamma(t)} \comma i,j=1,\dots,r,
\eeqa
where $\vep_i \defeq g(E_i,E_i) = \pm1$. Although \eqref{eqn:geod2} incorporates \eqref{eqn:geod}, do observe that this $\Aij(t)$ is not symmetric in general: $\Aij(t) = -A^{\scalebox{0.6}{\emph{$\gamma$}}}_{ji}(t)$ if $E_i$ and $E_j$ do not have the same causal character.  Nevertheless, we now define the plane wave limit of $(M,g)$ along $\gamma(t)$ just as we did in Definition \ref{def:PWL}:

\begin{defn}[semi-Riemannian plane wave limit]
\label{def:PWL2}
Let $(M,g)$ be a semi-Riemannian $n$-manifold \emph{(}$n\geq3$\emph{)} and $\gamma(t)$ a maximal geodesic with domain $I \subseteq \RR$.  The Lorentzian metric \emph{$\gL$} defined on $\RR \times I \times \RR^{r} \subseteq \RR^{r+2} = \{(v,t,x^1,\dots,x^{r})\}$ by
\beqa
\label{eqn:PWL2}
\text{\emph{$\gL$}} \defeq 2dvdt + \Big(\!\sum_{i,j=1}^{r}\!\Aij(t)x^ix^j\Big) (dt)^2 + \sum_{i=1}^{r} (dx^i)^2,
\eeqa
where
$$
r \defeq \left\{\begin{array}{ll}
n-1 &\text{if $\gamma'$ is spacelike or timelike},\\
n-2 &\text{if $\gamma'$ is lightlike},
\end{array}\right.
$$
with the $\Aij(t)$'s defined via any $g$-orthonormal frame $\{E_1,\dots,E_{r}\} $ parallel along $\gamma(t)$ as in \eqref{eqn:geod2}, is the \emph{plane wave limit of $(M,g)$ along $\gamma$.}
\end{defn}

Note that, despite the additional $\vep_i|\vep_j|$ factors in \eqref{eqn:geod2}, $\gL$ remains an isometric invariant of $(M,g)$ (recall Lemma \ref{lemma:iso}). Having said that, the curvature components of $\gL$ itself now satisfy
\beqa
\text{Rm}_{\scalebox{0.6}{$\gL$}}(\partial_i,\partial_t,\partial_t,\partial_j) \!\!&\overset{\eqref{eqn:psec}}{=}&\!\! -\frac{H_{ij}}{2}\nonumber\\
&\overset{\eqref{eqn:PWL2}}{=}&\!\! -\frac{1}{2}(\Aij(t)+A^{\scalebox{0.6}{\emph{$\gamma$}}}_{ji}(t))\nonumber\\
&=&\!\! \left\{\begin{array}{ll}
\vep_i|\vep_j|\,\text{Rm}_{\scalebox{0.6}{\emph{g}}}(E_i,\gamma',\gamma',E_j) &\text{if $\vep_i=\vep_j$},\\
0 &\text{if $\vep_i\neq \vep_j$}.
\end{array}\right.\label{eqn:some0}
\eeqa
This contrasts with the Riemannian case \eqref{eqn:crucial}. The reason for this our choice of $\Aij(t)$ will become clear in Proposition \ref{prop:conj} below. Indeed, although it is not immediately clear, what \eqref{eqn:geod2} has in fact distinguished is the following class of geodesics, a class that includes all geodesics on Riemannian manifolds as well as timelike and lightlike geodesics on Lorentzian manifolds:

\begin{defn}
\label{def:cc}
A geodesic $\gamma(t)$ in a semi-Riemannian manifold $(M,g)$ is \emph{causally independent} if there exists a parallel frame $\{E_1,\dots,E_r\}$ along $\gamma(t)$ as in Lemma \ref{lemma:r} such that \emph{$\text{Rm}_{\scalebox{0.6}{$g$}}(E_i,\gamma',\gamma',E_j)\big|_{\gamma(t)} = 0$} whenever $E_i,E_j$ do not have the same causal character.
\end{defn}

Observe that if $(M,g)$ is Riemannian, then all geodesics are causally independent by default; if $(M,g)$ is Lorentzian, then the same is true for all timelike and lightlike geodesics. However, spacelike geodesics in Lorentzian manifolds are not, in general, causally independent, nor are spacelike, timelike, or lightlike geodesics on semi-Riemannian manifolds of arbitrary signature. Now, if a geodesic $\gamma(t)$ is causally independent, and if $J$ is a Jacobi field along $\gamma(t)$ of the form $J(t) = \sum_{i=1}^rJ^i(t)E_i$ for some smooth functions $J^1(t),\dots,J^r(t)$, then the significance of Definition \ref{def:cc} can be seen by writing the Jacobi equation $J''+R_{\scalebox{0.6}{$g$}}(J,\gamma')\gamma'=0$ in matrix form, where we are assuming here that the first $k$ vectors in $\{E_1,\dots,E_r\}$ are the timelike ones:
$$
\begin{pmatrix}
\ddot{J}^1\\
\ddot{J}^2\\
\vdots\\
\ddot{J}^r\\
\end{pmatrix} = \begin{pmatrix} \underbrace{\,\text{Rm}_{\scalebox{0.6}{$g$}}(E_i,\gamma',\gamma',E_j)\,}_{\text{$k \times k$ matrix of timelike $E_i$'s}} & O\\ &  \\O & \underbrace{\,\text{Rm}_{\scalebox{0.6}{$g$}}(E_i,\gamma',\gamma',E_j)\,}_{\text{$(r-k) \times (r-k)$ matrix of spacelike $E_i$'s}}\end{pmatrix}\begin{pmatrix}
J^1\\
J^2\\
\vdots\\
J^r\\
\end{pmatrix}\cdot
$$
Another way of appreciating the ``block separability" shown here is via $\gamma(t)$'s index form. Indeed, let $\gamma\colon[a,b]\lra M$ be a causally independent geodesic segment of a semi-Riemannian manifold $(M,g)$ and consider its corresponding index form; let us assume here that $\gamma$ is unit-timelike.  Recall that the \emph{index form} of $\gamma$ is the symmetric bilinear form defined on the (infinite-dimensional) vector space
\beqa
V_0^{\perp} \defeq \{\text{piecewise smooth vector fields along $\gamma\,:\,$}&&\nonumber\\&&\hspace{-1in}g(V,\gamma') = 0\ , V(a) = V(b) = 0\}\nonumber
\eeqa
by 
$$
I(V,W) \defeq -\!\int_a^b \big(g(V',W') - \text{Rm}_{\scalebox{0.6}{$g$}}(V,\gamma',\gamma',W)\big)dt \comma V,W \in V_0^{\perp}.
$$
Writing $V = \sum_{i=1}^r V^iE_i$, we may then identify $V_t \defeq \sum_{i=1}^k V^iE_i$ and $V_s \defeq \sum_{i=k+1}^r V^iE_i$ as $V$'s ``timelike" and ``spacelike" parts, respectively; doing the same for $W$ yields $W = W_t+W_s$. Then the causal independence of $\gamma(t)$ is easily seen to yield
\beqa
I(V,W) \!\!&=&\!\! -\!\int_a^b\big(g(V'_t,W'_t) - \text{Rm}_{\scalebox{0.6}{$g$}}(V_t,\gamma',\gamma',W_t)\big)dt \nonumber\\
&&\ \ -\!\int_a^b\big(g(V'_s,W'_s) - \text{Rm}_{\scalebox{0.6}{$g$}}(V_s,\gamma',\gamma',W_s)\big)dt\nonumber\\
&=&\!\! I(V_t,W_t) + I(V_s,W_s),\label{eqn:2-to-1}
\eeqa
so that the index form necessarily separates into timelike and spacelike parts as well. This suggests that a Morse Index Theorem should hold for such geodesics\,---\,even given the pathological conjugate-point behavior that holds in the semi-Riemannian case in general, as was shown in \cite{helfer,piccione_conj}.  As we'll see in Proposition \ref{prop:conj} and Theorem \ref{thm:conj} below, this is indeed the case\,---\,and the virtue of the plane wave limit $\gL$ is that we can bypass the need to re-analyze the variational calculus involved in the Morse Index Theorem and instead pass over directly to its (already established) Lorentzian version (see \eqref{eqn:MIT} below). Before doing so, let us provide an example of a causally independent geodesic. Consider the following ``pp-wave-type" metric on $\RR^4 \defeq \{(v,t,x,y)\}$ with signature $(-\!-\!++)$ and $H(t,x,y)$ a smooth function:
$$
h \defeq 2dvdt + H(t,x,y)(dt)^2-(dx)^2+(dy)^2.
$$
The only difference between the Christoffel symbols of $h$ and those of the Lorentzian pp-waves \eqref{eqn:Ch} is the red plus sign appearing here:
$$
\cdh{\partial_x}{\partial_t} = \frac{H_x}{2}\partial_v\commas\cdh{\partial_y}{\partial_t} = \frac{H_y}{2}\partial_v\commas \cdh{\partial_t}{\partial_t} = \frac{H_t}{2}\partial_v\, \textcolor{red}{+}\, \frac{1}{2}H_x\partial_x - \frac{1}{2}H_y\partial_y.
$$
As a consequence, it is straightforward to verify that, while the only nonvanishing components of the Riemann curvature tensor of $h$ are, like \eqref{eqn:psec}, still
\beqa
\label{eqn:ssh}
\text{Rm}_{\scalebox{0.5}{\emph{h}}}(\partial_x,\partial_t,\partial_t,\partial_x) = -\frac{H_{xx}}{2} &\commas& \text{Rm}_{\scalebox{0.5}{\emph{h}}}(\partial_y,\partial_t,\partial_t,\partial_y) = -\frac{H_{yy}}{2},\nonumber\\
&&\hspace{-1in}\text{Rm}_{\scalebox{0.5}{\emph{h}}}(\partial_x,\partial_t,\partial_t,\partial_y) = -\frac{H_{xy}}{2},
\eeqa
the Ricci tensor now takes the form
\beqa
\label{eqn:newRicci}
\text{Ric}_{\scalebox{0.5}{\emph{h}}}(\partial_t,\partial_t) = -\frac{1}{2}(\textcolor{red}{-}H_{xx}+H_{yy}),
\eeqa
with all other components still vanishing.  Now set $H(t,x,y) = x^2-y^2$ and consider the curve $\gamma(s) =(0,s,0,0)$. Since $\gamma'(s) = \partial_t\big|_{\gamma(s)}$,
$$
h(\gamma',\gamma')\Big|_{\gamma(s)} = H\Big|_{\gamma(s)} = 0 \comma \cdh{\gamma'}{\gamma'} = \frac{1}{2}H_x\partial_x - \frac{1}{2}H_y\partial_y\Big|_{\gamma(s)} = 0,
$$
so that $\gamma(s)$ is a (complete) lightlike geodesic. Furthermore, $\partial_x$ and $\partial_y$ are, respectively, unit timelike and spacelike vector fields which, when restricted to $\gamma(s)$, are parallel along it:
$$
\cdh{\gamma'}{\partial_x} = \frac{H_x}{2}\partial_v\Big|_{\gamma(s)} = 0 = \frac{H_y}{2}\partial_v\Big|_{\gamma(s)} =\cdh{\gamma'}{\partial_y}.
$$
We may therefore take $\{E_1,E_2\} = \{\partial_x\big|_{\gamma(s)},\partial_y\big|_{\gamma(s)}\}$ in Lemma \ref{lemma:r}. An application of \eqref{eqn:ssh} thus shows that $\gamma(s)$ is causally independent:
$$
\text{Rm}_{\scalebox{0.5}{\emph{h}}}(\partial_x,\gamma',\gamma',\partial_x) = -1\commass \text{Rm}_{\scalebox{0.5}{\emph{h}}}(\partial_y,\gamma',\gamma',\partial_y) = +1 \commass \text{Rm}_{\scalebox{0.5}{\emph{h}}}(\partial_x,\gamma',\gamma',\partial_y) = 0.
$$
(This example generalizes easily to arbitrary dimension and signature.) Now notice one further property of $\gamma(s)$, namely, that it has conjugate points. Indeed, the vector field
$$
J(s) \defeq (\sin s)\partial_x\big|_{\gamma(s)}
$$
is a Jacobi field along $\gamma(s)$. This follows because $J''(s) = -(\sin s)\partial_x\big|_{\gamma(s)}$ and $R_{\scalebox{0.5}{\emph{h}}}(J,\gamma')\gamma'\big|_{\gamma(s)} = (\sin s)R_{\scalebox{0.5}{\emph{h}}}(\partial_x,\gamma')\gamma'\big|_{\gamma(s)}$, so that
$$
h(J''+R_{\scalebox{0.5}{\emph{h}}}(J,\gamma')\gamma',\partial_x) = -(\sin s)h_{xx} + (\sin s)\text{Rm}_{\scalebox{0.5}{\emph{h}}}(\partial_x,\gamma',\gamma',\partial_x) = 0,
$$ 
with all other components vanishing also. As we will show in Theorem \ref{thm:conj} below, this is no coincidence: The completeness of $\gamma(s)$, together with the fact that $\text{Ric}_{\scalebox{0.5}{\emph{h}}}(\gamma',\gamma') \geq 0$ (recall \eqref{eqn:newRicci}) while $\text{Rm}_{\scalebox{0.5}{\emph{h}}}(\cdot,\gamma',\gamma',\cdot) \neq 0$, \emph{guarantees} it.  This is, in fact, the semi-Riemannian generalization of the classical focusing result \cite{HP1970}. As we now show, it is also a natural application of our semi-Riemannian plane wave limit.

\subsection{Geodesic deviation of causally independent geodesics.}\label{subsec:proof}To do so, we should first address the reason for the inclusion of the factor ``$\vep_i|\vep_j|$" in \eqref{eqn:geod2}. As we now show, its inclusion is required so that $\gL$ has the same property as was shown to hold for Penrose's original limit in \cite{blau04}\,---\,namely, so that $\gL$ encodes $\gamma(t)$'s \emph{geodesic deviation} onto the geodesics of the plane wave $(\RR^{r+2},\gL)$. When $g$ is semi-Riemannian, the only caveat is that $\gamma$ must be causally independent:
\begin{prop}
\label{prop:conj}
Let $\gamma(t)$ be a causally independent geodesic of a semi-Riemannian $n$-manifold $(M,g)$ \emph{(}$n\geq3$\emph{)} and let \emph{$\gL$} be the plane wave limit of $(M,g)$ along $\gamma(t)$ as in Definition \ref{def:PWL2}. Then any geodesic $\tilde{\gamma}(s)$ of \emph{$\gL$} satisfying \emph{$\gL(\tilde{\gamma}',\partial_v) \neq 0$} has conjugate points if and only if $\gamma(t)$ does.
\end{prop}

\begin{proof}
Let $\{E_1,\dots,E_k,E_{k+1},\dots,E_{r}\}$ be a parallel frame along $\gamma(t)$ as in Definition \ref{def:cc}, with $\{E_1,\dots,E_k\}$ the timelike vectors (by Lemma \ref{lemma:r}, $k=\nu$ if $\gamma(t)$ is spacelike and $k=\nu-1$ if $\gamma(t)$ is timelike or lightlike, $\nu$ being the index of $g$).  We may express any normal Jacobi field $J(t)$ along $\gamma(t)$ as $J(t) = \sum_{i=1}^{r}J^i(t)E_i$ for some smooth functions $J^1(t),\dots,J^{r}(t)$. Now, let $\tilde{\gamma}(s)$ be a geodesic of $\gL$ satisfying $\gL(\tilde{\gamma}',\partial_v) \neq 0$.  As $\tilde{\gamma}^t(s)$ must be linear in $s$ (recall \eqref{eqn:t}), and as we are assuming that $\dot{\tilde{\gamma}}^t = \gL(\tilde{\gamma}',\partial_v) \neq 0$, we may rescale the affine parameter $s$ if necessary so that $\tilde{\gamma}^t(s) = s$; i.e., so that $\tilde{\gamma}(s)$'s domain coincides with the maximal domain $I \subseteq \RR$ of $\gamma(t)$.  Suppose that this $\tilde{\gamma}(s)$ has a pair of conjugate points; then there exists a nontrivial Jacobi field $\tilde{J}(s)$ along $\tilde{\gamma}(s)$ vanishing at two distinct points (and hence orthogonal to $\tilde{\gamma}'(s)$).  Expressing $\tilde{J}(s)$ as $\tilde{J}(s) = (\tilde{J}^v(s),\tilde{J}^t(s),\tilde{J}^1(s),\dots,\tilde{J}^{r}(s))$ in the coordinates \eqref{eqn:PWL2}, observe that $\tilde{J}^t(s) = \gL(\tilde{J},\partial_v)\big|_{\tilde{\gamma}(s)}$, too, must be linear in $s$, because $\partial_v$ is parallel:
$$
\ddot{\tilde{J}}^t(s) =\gL(\tilde{J}'',\partial_v)\Big|_{\tilde{\gamma}(s)} = -\text{Rm}_{\scalebox{0.6}{$\gL$}}(\tilde{J},\tilde{\gamma}',\tilde{\gamma}',\partial_v)\Big|_{\tilde{\gamma}(s)} = 0.
$$
Thus if $\tilde{J}(s)$ is to be zero at two distinct points, then $\tilde{J}^t(s) = 0$. Furthermore, if we use the fact that $\gL(\tilde{J},\tilde{\gamma}')\big|_{\tilde{\gamma}(s)} = 0$ to solve for $\tilde{J}^v$ in terms of $\tilde{J}^i,\dot{\tilde{\gamma}}^i$, then we arrive at the following expression for $\tilde{J}(s)$:
\beqa
\label{eqn:J0}
\tilde{J}(s) = -\Big(\sum_{i=1}^{r}\tilde{J}^i(s)\dot{\tilde{\gamma}}^i(s)\Big)\partial_v + \sum_{i=1}^{r}\tilde{J}^i(s) \partial_i\,\Big|_{\tilde{\gamma}(s)}.
\eeqa
In particular, this shows that the functions $\tilde{J}^1(s),\dots,\tilde{J}^{r}(s)$ cannot all vanish, as $\tilde{J}(s)$ is nontrivial.  With this in mind, and recalling that $\{E_1,\dots,E_k\}$ are the timelike vectors and $\{E_{k+1},\dots,E_r\}$ the spacelike ones, we now claim that the following two vector fields along $\gamma(t)$ in $(M,g)$,
\beqa
\label{eqn:newJg}
\tilde{J}_{\scalebox{0.6}{$t$}}(t) \defeq \sum_{j=1}^{k}\tilde{J}^j(t)E_j \comma \tilde{J}_{\scalebox{0.6}{$s$}}(t) \defeq \!\!\sum_{j=k+1}^{r}\!\!\tilde{J}^j(t)E_j,
\eeqa
are both Jacobi fields (by our remark above, note that they cannot both be trivial). Indeed, observe that each $\tilde{J}^j(t) = \tilde{J}^j(s)$ in \eqref{eqn:newJg} can be expressed as $\tilde{J}^j(s) = \gL(\tilde{J},\partial_j)\big|_{\tilde{\gamma}(s)}$; differentiating the latter along $\tilde{\gamma}(s)$, and noting via \eqref{eqn:Ch} that
$$
\nabla^{\scalebox{0.6}{$\gL$}}_{\!\tilde{\gamma}'(s)}\partial_j = \nabla^{\scalebox{0.6}{$\gL$}}_{\!\partial_t}\partial_j\Big|_{\tilde{\gamma}(s)} = \frac{1}{2}H_j(\tilde{\gamma}(s))\partial_v\Big|_{\tilde{\gamma}(s)},
$$
we have
$$
\dot{\tilde{J}}^j(s) = \gL(\tilde{J}',\partial_j) + \frac{1}{2}H_j(\tilde{\gamma}(s))\cancelto{0}{\gL(\tilde{J},\partial_v)}\Big|_{\tilde{\gamma}(s)} \comma j=1,\dots,r,
$$
where the second term on the right-hand sides vanishes because $\gL(\tilde{J},\partial_v) = \tilde{J}^t(s) = 0$.  A second derivative likewise yields $\ddot{\tilde{J}}^j(s) = \gL(\tilde{J}'',\partial_j)\big|_{\tilde{\gamma}(s)}$, using the fact that $\gL(\tilde{J}',\partial_v)\big|_{\tilde{\gamma}(s)} = \gL(\tilde{J},\partial_v)'\big|_{\tilde{\gamma}(s)} = 0$. The $\gL$-Jacobi equation for $\tilde{J}(s)$ now yields, for each $j=1,\dots,r$,
\beqa
\ddot{\tilde{J}}^j(s) \!\!&=&\!\! \gL(\tilde{J}'',\partial_j)\Big|_{\tilde{\gamma}(s)}\nonumber\\
&=&\!\! -\text{Rm}_{\scalebox{0.6}{$\gL$}}(\tilde{J},\tilde{\gamma}',\tilde{\gamma}',\partial_j)\Big|_{\tilde{\gamma}(s)}\nonumber\\
&=&\!\! -\sum_{i=1}^{r}\tilde{J}^i(s)\underbrace{\,\text{Rm}_{\scalebox{0.6}{$\gL$}}(\partial_i,\tilde{\gamma}',\tilde{\gamma}',\partial_j)\,}_{\text{$\text{Rm}_{\scalebox{0.6}{$\gL$}}(\partial_i,\partial_t,\partial_t,\partial_j)$}}\Big|_{\tilde{\gamma}(s)}\nonumber\\
&\overset{\eqref{eqn:some0}}{=}&\!\! -\sum_{i=1}^{r}\tilde{J}^i(s)\text{$\left\{\begin{array}{ll}
\vep_i|\vep_j|\,\text{Rm}_{\scalebox{0.6}{\emph{g}}}(E_i,\gamma',\gamma',E_j) &\text{if $\vep_i=\vep_j$},\\
0 &\text{if $\vep_i\neq \vep_j$}.
\end{array}\right.$}\label{eqn:vep_sum}
\eeqa
Consider now $\tilde{J}_{\scalebox{0.6}{$t$}}$ in \eqref{eqn:newJg}, for which $j \in \{1,\dots,k\}$, hence $\vep_j = -1$. For these $j$, the nonzero terms in \eqref{eqn:vep_sum} will arise when $\vep_i=-1$; i.e., when $i \in \{1,\dots,k\}$:
\beqa
\ddot{\tilde{J}}^j(s) \!\!&=&\!\! \sum_{i=1}^{k}\tilde{J}^i(s)\text{Rm}_{\scalebox{0.6}{$g$}}(E_i,\gamma',\gamma',E_j)\Big|_{\gamma(s)} \comma j = 1,\dots,k,\nonumber\\
&\overset{\eqref{eqn:newJg}}{=}&\!\! \text{Rm}_{\scalebox{0.6}{$g$}}(\tilde{J}_{\scalebox{0.6}{$t$}},\gamma',\gamma',E_j)\Big|_{\gamma(s)}.\label{eqn:last_t}
\eeqa
Likewise for $\tilde{J}_{\scalebox{0.6}{$s$}}$ in \eqref{eqn:newJg}, for which $j \in \{k+1,\dots,r\}$, so that $\vep_j = +1$. Now the nonzero terms in \eqref{eqn:vep_sum} will arise when $\vep_i=+1$, when $i \in \{k+1,\dots,r\}$:
\beqa
\ddot{\tilde{J}}^j(s) \!\!&=&\!\! -\!\!\sum_{i=k+1}^{r}\!\!\tilde{J}^i(s)\text{Rm}_{\scalebox{0.6}{$g$}}(E_i,\gamma',\gamma',E_j)\Big|_{\gamma(s)} \comma j = k+1,\dots,r,\nonumber\\
&\overset{\eqref{eqn:newJg}}{=}&\!\! -\text{Rm}_{\scalebox{0.6}{$g$}}(\tilde{J}_{\scalebox{0.6}{$s$}},\gamma',\gamma',E_j)\Big|_{\gamma(s)}.\label{eqn:last_s}
\eeqa
We can now verify that $\tilde{J}_{\scalebox{0.6}{$t$}}(t)$ and $\tilde{J}_{\scalebox{0.6}{$s$}}(t)$ are indeed Jacobi fields along $\gamma(t)$ in $(M,g)$. To begin with, the vector field $\tilde{J}_{\scalebox{0.6}{$t$}}''+R_{\scalebox{0.6}{$g$}}(\tilde{J}_{\scalebox{0.6}{$t$}},\gamma')\gamma'$ is $g$-orthogonal to $\gamma'$ because $\tilde{J}_{\scalebox{0.6}{$t$}}''$ and $R_{\scalebox{0.6}{$g$}}(\tilde{J}_{\scalebox{0.6}{$t$}},\gamma')\gamma'$ are, hence it lies in $\text{span}\{E_1,\dots,E_r\}$ (or $\text{span}\{E_1,\dots,E_r,\gamma'\}$ if $\gamma(t)$ is lightlike).  Verifying that $\tilde{J}_{\scalebox{0.6}{$t$}}''+R_{\scalebox{0.6}{$g$}}(\tilde{J}_{\scalebox{0.6}{$t$}},\gamma')\gamma' = 0$ therefore means verifying that $g(\tilde{J}_{\scalebox{0.6}{$t$}}''+R_{\scalebox{0.6}{$g$}}(\tilde{J}_{\scalebox{0.6}{$t$}},\gamma')\gamma' ,E_i) = 0$ for each $i=1,\dots,r$ (as well as with $\gamma'$ if $\gamma(t)$ is lightlike, though this latter verification is trivial, since $g(\tilde{J}_{\scalebox{0.6}{$t$}}'',\gamma') = g(\tilde{J}_{\scalebox{0.6}{$t$}},\gamma')'' = 0 =-\text{Rm}_{\scalebox{0.6}{$g$}}(\tilde{J}_{\scalebox{0.6}{$t$}},\gamma',\gamma',\gamma')$). To do so, we first call upon \eqref{eqn:last_t}, which tells us that for (timelike) $E_{1},\dots,E_k$,
$$
g(\tilde{J}_{\scalebox{0.6}{$t$}}'',E_j) \overset{\eqref{eqn:newJg}}{=} -\ddot{\tilde{J}}^j \overset{\eqref{eqn:last_t}}{=} -\text{Rm}_{\scalebox{0.6}{$g$}}(\tilde{J}_{\scalebox{0.6}{$t$}},\gamma',\gamma',E_j) \comma j = 1,\dots,k.
$$
For $E_{k+1},\dots,E_r$, we use the fact that $\gamma(t)$ is causally independent:
$$
g(\tilde{J}_{\scalebox{0.6}{$t$}}'',E_j) \overset{\eqref{eqn:newJg}}{=} 0  = -\text{Rm}_{\scalebox{0.6}{$g$}}(\tilde{J}_{\scalebox{0.6}{$t$}},\gamma',\gamma',E_j)\comma j = k+1,\dots,r.
$$
Thus $\tilde{J}_{\scalebox{0.6}{$t$}}(t)$ is indeed a Jacobi field along $\gamma(t)$; if it is nontrivial, then it vanishes at two distinct points\,---\,in fact at exactly the same affine parameter values as $\tilde{\gamma}(s)$. A similar analysis and conclusion follows for $\tilde{J}_{\scalebox{0.6}{$s$}}(t)$, via \eqref{eqn:last_s}. Conversely, if $J(t) = \sum_{i=1}^{r}J^i(t)E_i$ is a Jacobi field along $\gamma(t)$, then $J(t) = J_{\scalebox{0.6}{$t$}}(t)+J_{\scalebox{0.6}{$s$}}(t)$ as in \eqref{eqn:newJg}, and the reverse of the computation that led to \eqref{eqn:last_t}-\eqref{eqn:last_s} shows that \eqref{eqn:J0}, with $J^i(s)$ in place of $\tilde{J}^i(s)$,  will be a normal Jacobi field along any geodesic $\tilde{\gamma}(s)$ of $\gL$ satisfying $\tilde{\gamma}^t(s)=s$. In particular, $\tilde{\gamma}(s)$ will have conjugate points if $\gamma(t)$ does\,---\,at the same affine parameter values once again. Note that if $J(t) = \sum_{i=1}^{r}J^i(t)E_i$ is a Jacobi field along $\gamma(t)$ and the latter is causally independent, then at least one of $J_{\scalebox{0.6}{$t$}}(t)$ or $J_{\scalebox{0.6}{$s$}}(t)$ must be nontrivial. Thus the correspondence between Jacobi fields vanishing at two points is at most two-to-one: Up to two along $\gamma(t)$ for every one along $\tilde{\gamma}(s)$.
\end{proof}

Before using Proposition \ref{prop:conj} to generalize \cite{HP1970}, let us record the following easy corollary, which is immediate:

\begin{cor}
Let \emph{$\gL$} be the plane wave limit of a causally independent geodesic $\gamma$. If $\gamma$ has a pair of conjugate points, then all geodesics $\tilde{\gamma}$ of \emph{$\gL$} satisfying \emph{$\gL(\tilde{\gamma}',\partial_v) \neq 0$} have conjugate points.
\end{cor}

\subsection{Application to conjugate points} We now apply Proposition \ref{prop:conj} towards extending \cite{HP1970} (see also \cite[Prop.~12.10,12.17]{beem}) and \cite[Theorem~2]{Eschenburg} to all causally independent geodesics on semi-Riemannian manifolds:
\begin{thm}
\label{thm:conj}
Let $(M,g)$ be a semi-Riemannian $n$-manifold \emph{(}$n\geq 3$\emph{)} and $\gamma$ a complete causally independent geodesic. If $\gamma$ has no conjugate points and \emph{$\text{Ric}_{\scalebox{0.6}{$g$}}(\gamma',\gamma') \geq 0$}, then \emph{$\text{Rm}_{\scalebox{0.6}{$g$}}(\cdot,\gamma',\gamma',\cdot)\big|_{\gamma'^{\perp}} = 0$}. Furthermore, causally independent geodesic segments always have finitely many conjugate points.
\end{thm}

\begin{proof}
We will prove that if $\text{Rm}_{\scalebox{0.6}{$g$}}(\cdot,\gamma',\gamma',\cdot)\big|_{\gamma(0)^{\perp}} \neq 0$, then $\gamma(t)$ must have a pair of conjugate points. To begin with, given a parallel frame $\{E_1,\dots,E_r\}$ along $\gamma(t) $ such that $\text{Rm}_{\scalebox{0.6}{\emph{g}}}(E_i,\gamma',\gamma',E_j) = 0$ whenever $E_i$ and $E_j$ have different causal characters, the condition $\text{Rm}_{\scalebox{0.6}{$g$}}(\cdot,\gamma',\gamma',\cdot)\big|_{\gamma(0)^{\perp}} \neq 0$ implies the existence of a nonzero vector $X \defeq \sum_{i=1}^r a^iE_i\big|_{\gamma(0)} \in \gamma'(0)^\perp$ such that
\beqa
\label{eqn:prop2}
\text{Rm}_{\scalebox{0.6}{$g$}}(X,\gamma'(0),\gamma'(0),X) \neq 0.
\eeqa
Meanwhile, the Ricci tensor of $\gL$ satisfies
\beqa
\label{eqn:prop1}
\text{Ric}_{\scalebox{0.6}{$\gL$}}(\partial_t,\partial_t) \overset{\eqref{eqn:pRic}}{=} -\!\sum_{i=1}^{r} A_{ii}^{\scalebox{0.6}{$\gamma$}}(t) \overset{\eqref{eqn:geod2}}{=} \text{Ric}_{\scalebox{0.6}{$g$}}(\gamma',\gamma')\Big|_{\gamma(t)} \geq 0.
\eeqa
By Theorem \ref{thm:main}(vii), if $\gamma(t)$ is complete, then the plane wave limit $(\RR^{r+2},\gL)$ will be geodesically complete also. By Proposition \ref{prop:conj}, a geodesic $\tilde{\gamma}(s)$ of $\gL$ satisfying $\tilde{\gamma}^t(s) = s$ has conjugate points if and only if $\gamma(t)$ does. For any such geodesic $\tilde{\gamma}(s)$, if we set $\widetilde{X} \defeq \sum_{i=1}^r a^i\partial_i\big|_{\tilde{\gamma}(0)}$, where $a^1,\dots,a^r$ are the same coefficients as for $X$ above (note that $\widetilde{X}$ is $\gL$-spacelike), then
\beqa
\label{eqn:HP1}
\text{Rm}_{\scalebox{0.6}{$\gL$}}(\widetilde{X},\tilde{\gamma}',\tilde{\gamma}',\widetilde{X})\Big|_{\tilde{\gamma}(0)} \overset{\eqref{eqn:psec}}{=} \text{Rm}_{\scalebox{0.6}{$\gL$}}(\widetilde{X},\partial_t,\partial_t,\widetilde{X})\Big|_{\tilde{\gamma}(0)} \overset{\eqref{eqn:some0},\eqref{eqn:prop2}}{\neq} 0
\eeqa
and
\beqa
\label{eqn:HP2}
\text{Ric}_{\scalebox{0.6}{$\gL$}}(\tilde{\gamma}',\tilde{\gamma}')\Big|_{\tilde{\gamma}(s)} \overset{\eqref{eqn:pRic}}{=} \text{Ric}_{\scalebox{0.6}{$\gL$}}(\partial_t,\partial_t)\Big|_{\tilde{\gamma}(s)} \overset{\eqref{eqn:prop1}}{\geq} 0.
\eeqa
Now pick any timelike such geodesic $\tilde{\gamma}(s)$ with the additional property that $\tilde{\gamma}'(0)$ is $\gL$-orthogonal to $\widetilde{X}$. This fact, together with \eqref{eqn:HP1} and \eqref{eqn:HP2}, are exactly the conditions needed to ensure that $\tilde{\gamma}(s)$ will have a pair of conjugate points; see \cite[Proposition~12.10]{beem}. By Proposition \ref{prop:conj}, so must $\gamma(t)$ in $(M,g)$. Finally, the statement about finitely many conjugate points follows from the Lorentzian Morse Index Theorem.  Indeed, let $\tilde{\gamma}\colon[a,b]\lra \RR^{r+2}$ be any timelike geodesic segment of $\gL$ with $\tilde{\gamma}^t(s) = s$. Then its index form is given by
$$
I(V,W) \defeq -\!\int_a^b \big(\gL(V',W') - \text{Rm}_{\scalebox{0.6}{$\gL$}}(V,\tilde{\gamma}',\tilde{\gamma}',W)\big)ds \comma V,W \in V_0^{\perp}.
$$
Recall that the \emph{index of $\tilde{\gamma}$}, denoted $\text{Ind}(\tilde{\gamma})$, is defined to be the supremum over the dimensions of all subspaces $A \subseteq V_0^{\perp}$ on which the restriction $I\big|_{A\times A}$ is positive-definite. Next, for each $s_* \in (a,b)$, let $\text{dim}(\tilde{J}_{s_*})$ denote the dimension of the subspace spanned by all Jacobi fields along $\tilde{\gamma}(s)$ vanishing at $s=a$ and $s = s_*$; thus if $\text{dim}(\tilde{J}_{s_*}) = m > 0$, then $\tilde{\gamma}(s_*)$ is conjugate to $\tilde{\gamma}(a)$ with multiplicity $m$. The (timelike) Morse Index Theorem (see \cite[Theorem~10.27]{beem}) then says that 
$$
\text{Ind}(\tilde{\gamma}) = \!\!\sum_{s \in (a,b)}\!\! \text{dim}(\tilde{J}_{s}) \hspace{.2in}\text{and the latter sum is finite}.
$$
By Proposition \ref{prop:conj}, we thus have the following ``Morse Index Theorem" for causally independent geodesics $\gamma\colon[a,b]\lra M$ of semi-Riemannian manifolds: If we denote by $\text{dim}(J_{t_*})$ the dimension of the subspace spanned by all Jacobi fields along $\gamma(t)$ vanishing at $t=a$ and $t = t_*$, and if we recall once again that $s=t$, then
\beqa
\label{eqn:MIT}
\underbrace{\,\text{dim}(J_{t}) \leq 2\,\text{dim}(\tilde{J}_{t})\,}_{\text{for each $t \in (a,b)$, by Proposition \ref{prop:conj}}} \imp \sum_{t \in (a,b)} \text{dim}(J_{t}) \leq 2\,\text{Ind}(\tilde{\gamma}).
\eeqa
Thus the number of conjugate points along $\gamma(t)$ in $(M,g)$, counting multiplicity, is not only finite, but at most twice the index of any timelike geodesic segment of the plane wave limit $\gL$ of $\gamma(t)$.
\end{proof}

Note that the factor of two in \eqref{eqn:MIT} is to be expected, given our observation about the splitting of the index form in \eqref{eqn:2-to-1}. Finally, it is not hard to show that if $\gamma$ is spacelike or timelike (but not lightlike), then the conclusion of Theorem \ref{thm:conj} can be strengthened, to $\text{Rm}_{\scalebox{0.6}{$g$}}(\cdot,\gamma',\gamma',\cdot) = 0$.

\section*{Acknowledgments}
ABA thanks Miguel Angel Javaloyes with helpful discussions on Section \ref{subsec:geom}, Paolo Piccione with helpful discussions on Section \ref{subsec:proof}, Miguel S\'anchez for helpful discussions on Section \ref{subsec:sanchez}, and Matthias Blau for very helpful discussions throughout the paper.  He also warmly acknowledges the hospitality of the Albert Einstein Center at Universt\"at Bern.

\section*{References}
\renewcommand*{\bibfont}{\footnotesize}
\printbibliography[heading=none]
\end{document}